%% file: main.tex
\documentclass[a4paper,11pt]{article}

\usepackage{mathtools}
\usepackage{float}
\usepackage{enumitem}
\usepackage{graphics}
\usepackage{graphicx}
\usepackage{epsfig}
\usepackage{tikz}
\usepackage[T1]{fontenc}
\usepackage{amsmath}
\usepackage{amsthm}
\usepackage{algorithm}
\usepackage{algorithmic}
\usepackage{subfig}
\usepackage{url} 
\usepackage{makeidx} 
\usepackage{amsfonts} 
\usepackage{verbatim} 
\usepackage{cite} 
\usepackage{hyperref} 
\usepackage{array} 
\hypersetup{
    pdftitle={Technical Note},    
    colorlinks=true,       
	allcolors=blue
}

\usepackage{color}

\newtheorem{theorem}{Theorem}

\newtheorem{lemma}{Lemma}
\newtheorem{definition}{Definition}
\newtheorem{assumption}{Assumption}

\makeatletter

\newcommand{\Rmnum}[1]{\expandafter\@slowromancap\romannumeral #1@}
\makeatother

\begin{document}

\title{Distribution Systems Hardening against Natural Disasters}
%
%
%

\author{Yushi Tan, Arindam~K.~Das, Payman Arabshahi, and Daniel~S.~Kirschen}
\date{}

\maketitle

\begin{abstract}
Distribution systems are often crippled by catastrophic damage caused by a natural disaster. Well-designed hardening can significantly improve the performance of post-disaster restoration operations. Such performance is quantified by a resilience measure associated with the operability trajectory. The distribution system hardening problem can be formulated as a two-stage stochastic problem, where the inner operational problem addresses the proper sequencing of post-disaster repairs and the outer problem the judicious selection of components to harden. We propose a deterministic robust reformulation with two solution methods, an MILP formulation and a heuristic approach. We provide computational evidence on various IEEE test feeders which illustrates that the heuristic approach provides near-optimal hardening solutions efficiently.
\end{abstract}

\section{Introduction}

Natural disasters have caused major damage to electricity distribution networks and deprived homes and businesses of electricity for prolonged periods, for example Hurricane Sandy in November 2012~\cite{nerc2014sandy}, the Christchurch Earthquake in February 2011~\cite{eidinger2012christchurch} and the June 2012 Mid-Atlantic and Midwest Derecho~\cite{doe2012derecho}. Estimates of the annual cost of power outages caused by severe weather between 2003 and 2012 range from \$18 billion to \$33 billion on average\cite{eotp2013resiliency}. Physical damage to grid components must be repaired before power can be restored~\cite{gridwise2013resilience,  nerc2014sandy}. On the operational side, approaches have been proposed for scheduling the available repair crews in order to minimize the cumulative duration of customer interruption, which reduces the harm done to the affected community~\cite{tan2017scheduling, nurre2012restoring, coffrin2014transmission}. On the planning side, Kwasinski et al. \cite{eidinger2012christchurch} reported that facilities that had been upgraded or hardened in Christchurch at a cost of \$5 million, remained serviceable immediately after the September 2010 earthquake and saved approximately \$30 to \$50 million in subsequent repairs. Hardening minimizes the potential damages caused by disruptions, thereby facilitating restoration and recovery efforts, and the time it takes for the infrastructure system to resume operation~\cite{omer2013resilience}. However, as indicated in ~\cite{rollins2007hardening}, the difficulty of hardening does not lie in the design or construction of a hardened system, rather in the ability to quantify the expected performance improvement so that rational decisions can be made regarding increased cost versus potential future benefit. 
\subsection{Concept and quantification of resilience}
Resilience in infrastructure systems under natural disasters is an important current area of research. While several definitions of resilience have been proposed~\cite{mili2011taxonomy, bruneau2003framework, orourke2007critical, EPRI:2013tuba}, infrastructure resilience is typically defined as the ability to anticipate, prepare for, adapt to changing climate conditions and withstand, respond to, and recover rapidly from disruptions~\cite{obama2013preparing}. Resilience usually addresses the following four aspects - preparedness, robustness, resourcefulness and recovery, as illustrated in Fig.~\ref{fig:construct}, which is adapted from the `Resilience Construct' in~\cite{niac2010framework}. Since resourcefulness mainly depends on people instead of technology, the planning aspect, preparedness, should be guided by the other two operational aspects, robustness and recovery. 
\begin{figure}[htbp]
\centering
\includegraphics[trim = 100 120 100 250 ,clip, width = 1.0\columnwidth]{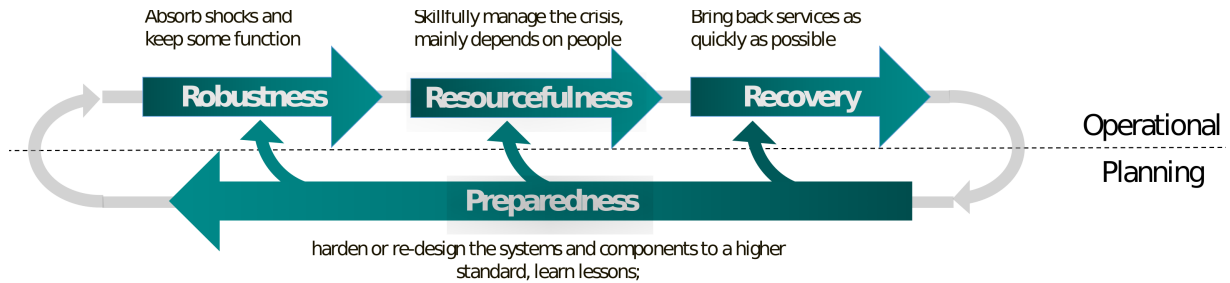}
\caption{Interactions between the four aspects of resilience}
\label{fig:construct}
\end{figure}

In a civil engineering context, resilience can be illustrated using the ``operability trajectory'', $Q(t)$, as shown in Figure~\ref{fig:operabtraj}, adopted  from~\cite{Reed:2009ekbaca}. The trajectory shows the increase in infrastructure functionality over time and is an effective visual indicator of the `goodness of the restoration process'. Robustness is quantified by the depth of functionality drop at time zero\footnote{Without any loss of generality, we assume that the restoration process commences at time $t=0$.}, while the quality of the recovery process is quantified by the ramp up time of the operability trajectory to full/satisfactory  functionality, post time zero. Obviously, we desire that an infrastructure system exhibit a relatively small drop in functionality at time zero and a quick ramp up time to full/satisfactory functionality, post time zero. Consequently, the ideal operability trajectory is defined by $Q_{ideal}(t) = 1, \, \forall t \geq 0$, assuming that operability is measured in fractional units instead of percentages. These two metrics can naturally be combined into an unifying measure of resilience~\cite{orourke2007critical}. Letting $T$ be some restoration time horizon, a resilience measure, $R$, can be defined as follows~\cite{Reed:2009ekbaca}: 
 \begin{align}
    R = \int_{0}^{T} Q(t) dt,
 \label{resilience_defn_civil}
 \end{align}
The closer $Q(t)$ is to $Q_{ideal}(t)$, the greater is the area under $Q(t)$, and therefore the greater is the resilience measure. 

Instead of maximizing the resilience measure defined in eqn.~\ref{resilience_defn_civil}, we could choose to minimize the quantity $\int_{0}^{T} Q_{ideal}(t) dt - \int_{0}^{T} Q(t) dt$, which is the area over the $Q(t)$ curve, bounded from above by $Q_{ideal}(t)$. This area, informally, the `other side of resilience', can be interpreted as a measure of `aggregate harm'. In a power system, it can be shown using the Lebesgue integral that minimizing this area is equivalent to minimizing the quantity $\sum_n w_n T_n$, where $w_n$ can be interpreted as the contribution of node $n$ to the overall loss in functionality of the system (or alternately, harm suffered by node $n$) and $T_n$ is the time to restore node $n$. Therefore, our objective for operational problems is to minimize the measure $\sum_n w_n T_n$, given a specific disaster scenario, while the objective for planning problems is to minimize $\sum_n w_n T_n$ in an expected sense, where the expectation is over all possible disaster scenarios.

\begin{figure}[htbp]
\centering
\includegraphics[width = 1.0\columnwidth]{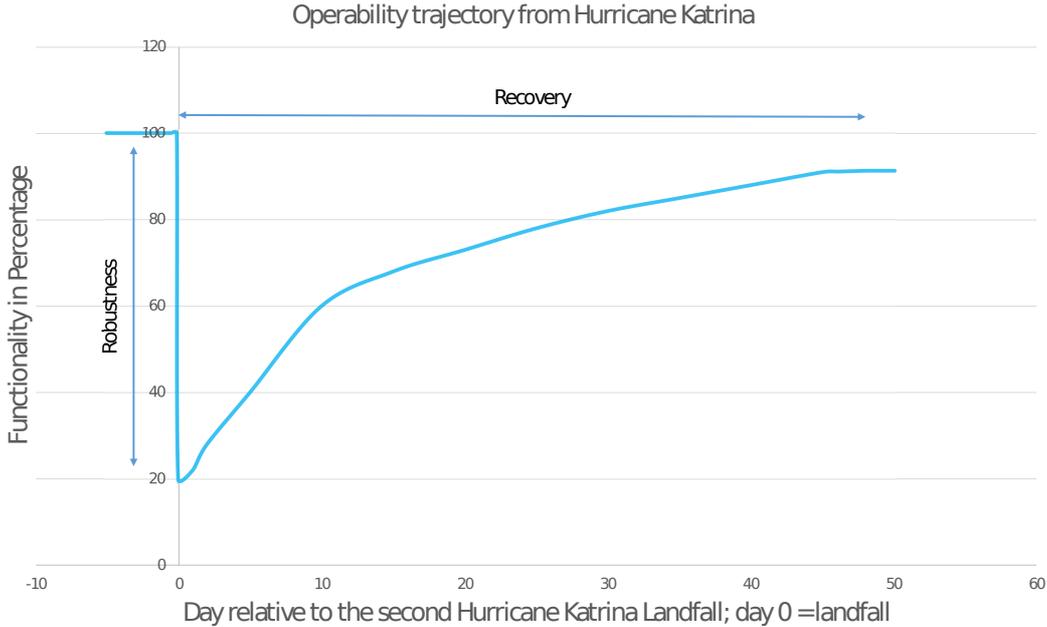}
\caption{Operability trajectory after Hurricane Katrina~\cite{Reed:2009ekbaca}.}
\label{fig:operabtraj}
\end{figure}
\subsection{Literature review}
\label{subsec:literature}
In recent years, several researchers have investigated different methods for hardening, but most focus solely on the robustness, i.e., worst-case load shed at the onset of disaster. Of note, a resilient distribution network planning problem (RDNP) was proposed in~\cite{yuan2016robust} to coordinate the hardening and distributed generation resource allocation. A tri-level defender-attacker-defender model is studied, in which the defender (hardening planner) selects a network hardening plan in the first stage, the attacker (natural disaster) disrupts the system with an interdiction budget, and finally, the defender (the distribution system operator) reacts by controlling DGs and switches in order to minimize the shed load. This model is improved in~\cite{ma2016resilience} by considering the investment cost and by eliminating the assumption that enhanced components should remain intact during any disaster scenario. Another direction of research enforces chance constraints on the loss of critical loads and normal loads respectively~\cite{yamangil2015resilient, nagarajan2016optimal}. A two-stage stochastic program and heuristic solution of hardening strategy were proposed in~\cite{romero2015seismic}, specifically for earthquake hazards, under the assumption that the repair times for similar types of components follow an uniform distribution, which simplifies the problem to a certain extent.
\subsection{Our approach}
To the best of our knowledge, this paper is the first to consider the restoration process in conjunction with hardening. Our approach can be seen as a two-stage stochastic problem. The first stage selects from the set of potential hardening choices and determines the extent of hardening to maximize the expected resilience measure $R$, while the second stage solves the operational problem in each possible scenario by optimizing the sequence of repairs given the hardening results. In the operational problem, we consider sequencing post-disaster repairs in distribution network with one repair crew.  Wang et al.  ~\cite{wangresearch} make a distinction between hardening activities and resiliency activities which are focused on the effectiveness of humans post-disaster. By assuming only one repair crew, we focus on the effects of network structure and components, and reduce the reliance on resourcefulness (i.e. the number of repair crews available). This issue will be discussed in more detail in Section~\ref{sec:solution}. Since an ideal formulation of the problem is hard to solve, we developed a reformulation and several heuristics techniques to solve the hardening problem.

The rest of the paper is organized as follows. In Section~\ref{sec:sqnc}, we  briefly review the problem of sequencing post-disaster repairs in distribution networks with one repair crew and discuss an MILP formulation and an optimal heuristic algorithm. In Section~\ref{sec:prob}, we formulate the problem of distribution system hardening against natural disasters, followed by a deterministic robust reformulation in Section~\ref{sec:reform}. Next, we motivate why we believe it is important to consider the restoration process in the hardening problem, the so called `restoration process aware hardening problem'. Two solution methods, an MILP formulation and an iterative heuristic  algorithm, are then discussed in Section~\ref{sec:solution}. The performance of these methods is validated by various case studies on standard IEEE test feeders in Section~\ref{sec:case}.
\section{Sequencing post-disaster repairs in distribution networks}
\label{sec:sqnc}
In this section, we briefly review the problem of sequencing post-disaster repairs in distribution networks with a single repair crew. Further details, including scheduling with multiple repair crews, can be found in~\cite{tan2017scheduling}.
\subsection{Distribution networks modeling}
A distribution network can be modeled by a graph $G$ with a set of nodes $N$ and a set of edges  $L$. Let $S \subset N$ represent the set of source nodes which are initially energized and $D = N \setminus S$ represent the set of sink nodes where consumers are located. An edge in $G$ represents a distribution feeder or some other connecting component. We assume that the network topology $G$ is radial, which is a valid assumption for many electricity distribution networks. Instead of a rigorous power flow model, we model network connectivity using a simple network flow model, i.e., as long as a sink node is connected to the source, we assume that all the loads connectd to this node can be supplied without violating any security constraint. For simplicity, we treat the three-phase distribution network as if it were a single-phase system. Our analysis could be extended to a three-phase system using a multi-commodity flow model, as in \cite{yamangil2014designing}. 
\subsection{Damage modeling}
Let $L^{D}$ and $L^{I} = L \setminus L^D$ denote the sets of damaged and intact edges, respectively. Without loss of generality, we assume that there is only one source node in $G$. If an edge is damaged, all downstream nodes lose power due to lack of electrical connectivity. Each damaged edge $l \in L^D$ has a (potentially) unique repair time $p_l$. At the operational stage, we assume perfect knowledge of the set $L^{D}$ and the corresponding repair times, while, at the planning stage, the repair times are modeled as random variables following some probability distribution.
\subsection{Sequencing with soft precedence constraints}
Let $w_{n}$ be a nonnegative quantity that captures the importance of the load at node $n$ and $T_{n}$ the time required to restore power at node $n$ (i.e. the energization time of node $n$). The importance of a node depends on the amount of load connected to it as well as the  types of load served. For example, re-energizing a hospital would receive a higher priority than a similar amount of residential load. We assume that crew travel times are minimal and can be either ignored or factored into the component repair times. Within this framework, our goal is to find a sequence by which the damaged edges should be repaired such that the aggregate harm $\sum_{n \in N}w_{n}T_{n}$ is minimized. 

We construct two simplified directed radial graphs to model the effect that the topology of the distribution network has on scheduling. The first graph, $G^{\prime}$, is called the `damaged component graph'. All nodes in $G$ that are connected by intact edges are merged into supernodes in $G^{\prime}$. The set of edges in $G^{\prime}$ is the set of damaged edges in $G$, $L^D$. The second graph, $P$, called a `soft precedence constraint graph', is formally defined in~\cite{tan2017scheduling}. An edge exists between two nodes in $P$ if they share the same node in $G^{\prime}$. The direction of the edge is determined by the hierarchy of components, or equivalently, the direction of power flow. A node in $P$ represents (1) a damaged edge $l$ in $G$ and (2) a set of nodes that could be energized if edge $l$ and all its predecessors are repaired. See Fig.~\ref{fig:sim_dag} for an illustration on the IEEE $13$-node test feeder (Fig.~\ref{fig:ieee13}). Additional details can be found in~\cite{tan2017scheduling}.
%
\begin{figure}[htbp]
\centering
\includegraphics[width=0.5\columnwidth]{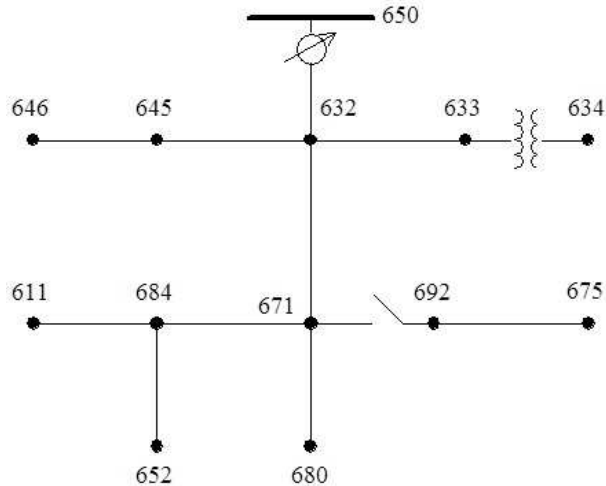}
\caption{IEEE 13 Node Test Feeder}
\label{fig:ieee13}
\end{figure}
\begin{figure}[htbp]
\centering
\subfloat[$G'$ graph]{\includegraphics[width=0.48\columnwidth]{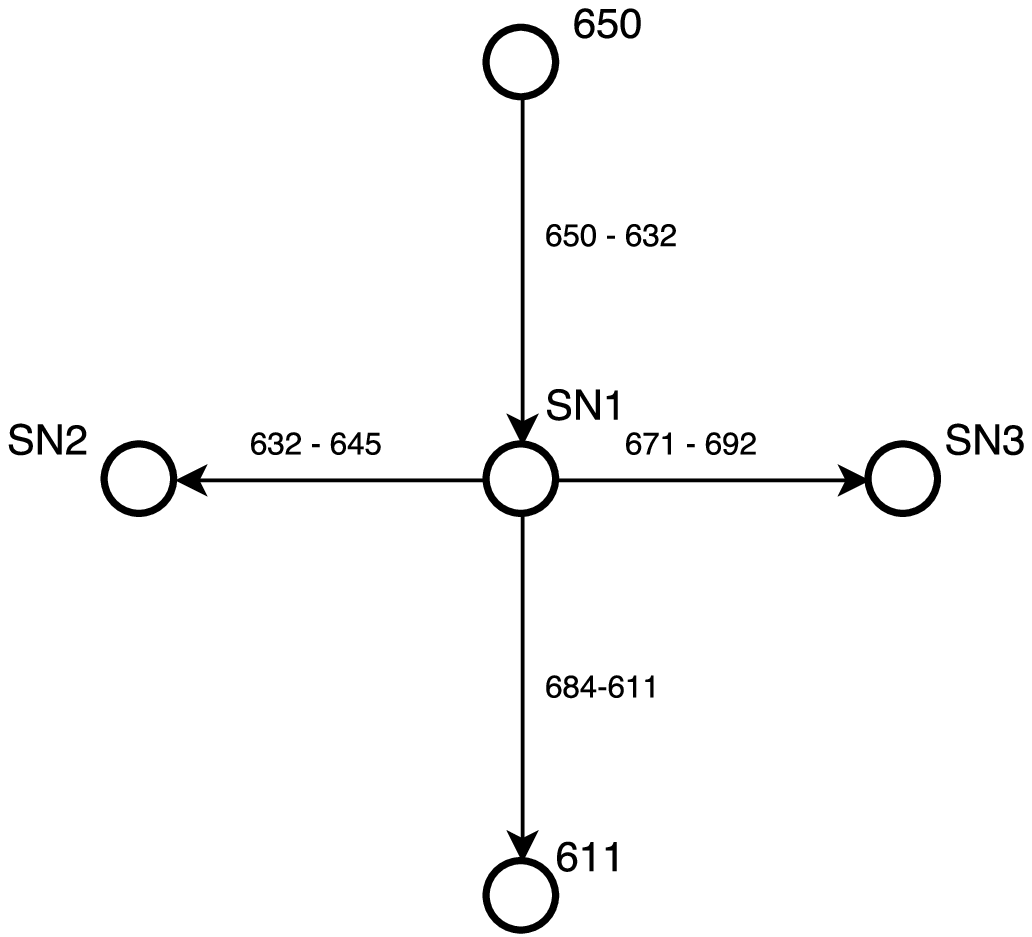}} \quad
\subfloat[$P$ graph]{\includegraphics[width=0.48\columnwidth]{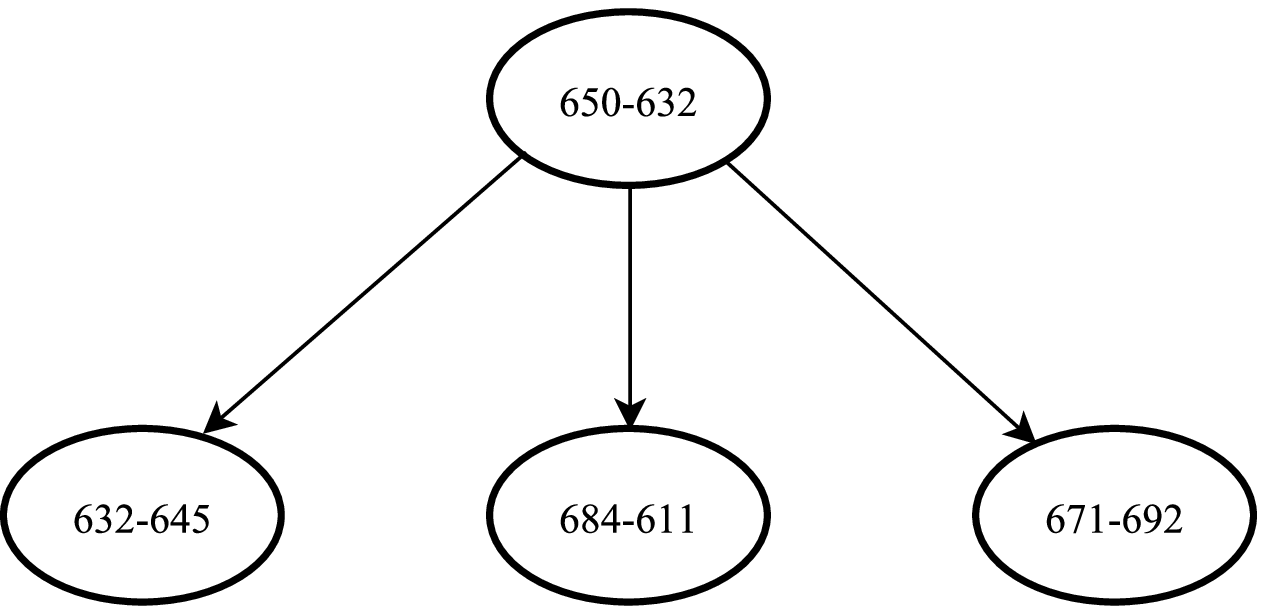}}
\caption{(a) The damaged component graph, $G^{\prime}$, obtained from Fig.~\ref{fig:ieee13}, assuming that the damaged edges are $650-632$, $632-645$, $684-611$ and $671-692$. (b) The corresponding soft precedence graph, $P$.}
\label{fig:sim_dag}
\end{figure}
\subsection{MILP formulation with one repair crew}
\label{sec_sequencing_MILP}
With only one repair crew, the damaged components must be repaired one by one, so there can be $L^D$ decisions to make, one at each time stage. The duration of each stage depends on the repair time of the component. We use two sets of binary decision variables. The first set of decision variables is denoted by $\{x_l^t\}$,  where $x_l^t = 1$ if edge $l$ is repaired at time stage $t$ and is equal to $0$ otherwise. The second set of decision variables is denoted by $\{u_i^t\}$, where $u_i^t = 1$ if node $i$ is energized at the end of time stage $t$ and is equal to $0$ otherwise. Let $T$ denote some restoration time horizon and $h^t$ denote the harm till time stage $t$. Although we cannot know $T$ exactly until the problem is solved, a conservative estimate is adequate. The MILP model for minimizing the aggregate harm is shown below:
\begin{subequations}
\begin{align}
\underset{x, u}{\text{min}} \quad &\sum_{t=1}^{T} h^t  \\
\text{s.t.} \quad 
& u_i^0 = 1, \; \forall i \in S \label{eqn:uinit}\\
& \sum_{t=1}^T u_i^t = 1, \; \forall i \in D \label{eqn:ufinal} \\
& \sum_{l \in L^D} x_l^t = 1, \; \forall t \in [1,T] \label{eqn:edgechosen} \\
& \sum_{\tau = 0}^{t - 1} u_i^{\tau} + u_j^t - 2 x_l^t \geq 0, j \in D, i \in Ne(j), \forall t \in [1,T] \label{eqn:uandx} \\
& d^0 = 0 \label{eqn:dinit} \\
& d^t \geq d^{t - 1} + p_l \times x_l^t, \; \forall t \in [1,T], \; \forall l \in L^D \label{eqn:defd} \\
& h^t \geq w_i r_i^t, \; \forall i \in D, \; \forall t \in [1,T] \label{eqn:defh} \\
& r_i^t \leq M \times u_i^t, \; \forall t \in [1,T], \; \forall i \in D \\
& r_i^t \leq d^t, \; \forall t \in [1,T], \; \forall i \in D \\
& r_i^t \geq 0 \\
& r_i^t \geq d^t - M \left(1 - u_i^t\right), \;\forall t \in [1,T], \; \forall i \in D \label{eqn:linearizeh} 
\end{align}
\end{subequations}
The first set of constraints binds the two sets of decision variables. Constraints (\ref{eqn:uinit}) and (\ref{eqn:ufinal}) specify that all source nodes be energized initially and all sink nodes be energized by time $T$. Constraint (\ref{eqn:edgechosen}) requires that only one damaged edge be chosen for repair at any time stage. Constraint (\ref{eqn:uandx}) requires that whenever some edge $l \in L^D$ is chosen for repair at time stage $t$, i.e., $x_l^t = 1$, it forces $u_j^t$ and $\sum_{\tau = 0}^{t - 1} u_i^{\tau}$ to be 1, i.e., node $j$ (the tail node of edge $l$) is energized at time stage $t$ and at least one of the neighbors of node $j$ in the graph $G$, which we denote by $Ne(j)$, is energized at some previous time stage. Implicitly, this constraint forces that an edge be chosen for repair at time stage $t$ if and only if the tail node of that edge can be immediately energized at the end of time stage $t$, which must be true when only one repair crew is available.

The second set of constraints connects the aggregate harm with the decision variables. The intermediate variable $d^t$ captures the aggregate restoration time just prior to time stage $t$. Constraint~(\ref{eqn:dinit}) initializes the aggregate restoration time to $0$. Constraint~(\ref{eqn:defd}) requires that the difference $d^t - d^{t-1}$, for some $t$, be at least the repair time of the edge being repaired at time $t$. If $l$ is the edge being repaired at time stage $t$, whose tail node is $j$, the $t^{th}$ stage harm can be modeled as follows: $h^t \geq w_j \left(u_j^t d^t\right)$. Constraints (\ref{eqn:defh}) to (\ref{eqn:linearizeh}) linearize the above nonlinear inequality using the big-$M$ method and the intermediate variable $r_i^t$. 
\subsection{Heuristic algorithm}
\label{subsec:seqheur}
We begin with the following lemma, which states that sequencing post-disaster repairs in distribution networks with a single repair crew is equivalent to the scheduling problem $1 \mid outtree \mid \sum_{n \in N} w_{n}C_{n}$, for which an optimal algorithm exists \cite{adolphson1973optimal}. A proof of this lemma can be found in~\cite{tan2017scheduling}.
\begin{lemma}
Single crew repair and restoration scheduling in distribution networks is equivalent to $1 \mid outtree \mid \sum_j w_{j}C_{j}$, where the outtree precedences are given in the soft precedence constraint graph $P$.
\label{lemma:singleeqiv}
\end{lemma}
Next, we briefly discuss the algorithm. Details and proofs can be found in~\cite{brucker2007scheduling} and~\cite{tan2017scheduling}. Let $J \subseteq L^D$ denote any subset of damaged edges. Define:
\begin{equation*}
w\left(J\right) := \sum_{j \in J} w_j, \ p\left(J\right) := \sum_{j \in J} p_j, \  q\left(J\right) := \frac{w\left(J\right)}{p\left(J\right)}
\end{equation*}
Algorithm \ref{alg:outree_merge} finds the optimal repair sequence by recursively merging the nodes in the soft precedence graph $P$. The input to this algorithm is the precedence graph $P$. Let $N(P) = \left\{1,2, \dots \vert N(P) \vert\right\}$ denote the set of nodes in $P$ (representing the set of damaged edges, $L^D$), with node $1$ being the designated root. Broadly speaking, at each iteration, a node $j \in N(P)$ (which could also be a group of nodes) is chosen to be merged into its immediate predecessor $i \in N(P)$ if $q(j)$ is the largest. The algorithm terminates when all nodes have been merged into the root. Upon termination, the optimal single crew repair sequence can be recovered from the predecessor vector and the element $A(1)$, which indicates the last job finished. It is shown in Section 4.3.1 of \cite{brucker2007scheduling} that the algorithm can compute the optimal sequence in $O(n \, \text{log} \, n)$-time. 
%
%
\begin{algorithm}
\caption{Optimal algorithm for single crew restoration in distribution networks. The input to this algorithm is the precedence graph $P$. The notation $\text{pred}(n)$ denotes the predecessor of any node $n \in P$.} 
\label{alg:outree_merge} 
\begin{algorithmic}[1]
    \STATE $w(1) \leftarrow -\infty$; \quad $\text{pred}(1) \leftarrow 0$;
    \FOR{$n = 1$ to $\vert N(P) \vert$}
    \STATE $A(n) \leftarrow n; \quad B_n \leftarrow \{n\}; \quad q(n) \leftarrow w(n)/p(n)$;
    \ENDFOR
    \FOR{$n = 2$ to $\vert N(P) \vert$}
    \STATE $\text{pred}(n) \leftarrow \text{parent of $n$ in $P$}$;
    \ENDFOR
    \STATE $\text{nodeSet} \leftarrow \{1,2, \cdots, \vert N(P) \vert\}$;
    \WHILE{$\text{nodeSet} \neq \{1\}$}
    \STATE Find $j \in \text{nodeSet}$ such that $q(j)$ is largest; \quad \text{$\%$ ties can be broken arbitrarily}
    \STATE Find $i$ such that $\text{pred}(j) \in B_{i}$, $i = 1,2,\dots \vert N(P) \vert$;
    \STATE $w(i) \leftarrow w(i) + w(j)$;
    \STATE $p(i) \leftarrow p(i) + p(j)$;
    \STATE $q(i) \leftarrow w(i)/p(i)$;
    \STATE $\text{pred}(j) \leftarrow A(i)$;
    \STATE $A(i) \leftarrow A(j)$;
    \STATE $B_i \leftarrow \{B_i, B_j\}$; \ \text{$\%$ `$,$' denotes concatenation} 
    \STATE $\text{nodeSet} \leftarrow \text{nodeSet} \setminus \{j\}$; 
    \ENDWHILE
\end{algorithmic}
\end{algorithm}
\section{The Hardening Problem: Formulation}
\label{sec:prob}
%
%
\subsection{Damage modeling}
As mentioned above, damages are modeled by repair time vectors associated with network components. Since no \textit{a priori} exact information about the damages  is available at the planning stage, we model the repair times as a random vector $\vec{\mathcal{P}}$. The uncertainties are twofold: the possible scenarios of natural disasters that planners want to take into account and the uncertain damages to components caused by a specific disaster. The distribution of $\vec{\mathcal{P}}$ can be a mixture of a Bernoulli distribution which represents the probability of damage and a (possibly) continuous distribution of repair time, such as the exponential~\cite{patton1979probability} or log-normal distribution~\cite{billinton1985distributional}. Mixed distributions, usually do not admit a closed-form expression of their distribution functions. In our work, we do not assume any knowledge of the distribution function, except for knowledge of the first moment $\mathbb{E}[\vec{\mathcal{P}}]$. 

Some planners tend to use the sample average approximation (SAA) methods by considering a limited set of component damage scenarios, which are either defined by users or drawn from a probabilistic model, as in~\cite{yamangil2015resilient, nagarajan2016optimal}. It is known that SAA methods converge to the optimal solution as the sample size goes to infinity. However, SAA methods require that the selected scenarios be typical and right on target, or the sample averaging needs to be performed over a large number of cases.
\subsection{Hardening options and costs}
In practice, multiple hardening actions are usually available for each network component. For example, hardening an edge can involve some combination of vegetation management, pole reinforcement, undergrounding, enhanced pole guying, ~\cite{ma2016resilience}. 
Typically, the goal of hardening a component is to lower the probability of its failure in the event of a disaster. However, since we are interested in maximizing the resilience of the system, or equivalently, minimizing the aggregate harm, simply lowering the probability of failure of a component is not sufficient. Since the aggregate harm is a function of the restoration times of the nodes, which in turn depend on the repair times of the damaged components (and the repair schedule), hardening a component can only be beneficial if it leads to a corresponding reduction in the repair time of that component.

In this paper, we assume that there is a finite set of hardening strategies for each edge $l$, which we denote by $K_l$. Each such strategy can be some combination of several disjoint hardening actions. We require that the hardening process select one strategy from the set $K_l$. Let $\vec{p} = \{p_l\}$, where $p_l$ is the `expected repair time' of component $l$ before hardening, $\Delta \vec{p} = \{\Delta p_{lk}\}$, where $\Delta p_{lk}$ is the `expected reduction in the repair time' of component $l$ due to hardening strategy $k \in K_l$, and $c_{lk}$ be the cost of implementing hardening strategy $k$ on edge $l$. We make the following assumption on the relationship between $c_{lk}$ and $\Delta p_{lk}$: 
\begin{assumption}
For any two hardening strategies $(k_1, k_2) \in K_l$, if $\Delta p_{l k_1} < \Delta p_{l k_2}$, then $c_{l k_1} < c_{l k_2}$ and vice versa.
\label{asmp:mono}
\end{assumption}
Generally, the more a component is hardened, the greater is the cost of hardening, but so is the reduction in repair times. The reasoning behind Assumption~\ref{asmp:mono} is similar to that of Proposition~1 in~\cite{sinha1979multiple}. If there exists two hardening strategies $(k_1, k_2) \in K_l$ which violate the assumption, i.e., $\Delta p_{l k_1} > \Delta p_{l k_2}$ is true while $c_{l k_1} < c_{l k_2}$, strategy $k_2$ cannot be part of the optimal hardening solution. 
%
%
\subsection{Ideal stochastic programming model}
\begin{definition}
Given a repair time vector $\vec{p}$, the min-harm (or equivalently, max-resilience) function, denoted by $f(\cdot)$, is the mapping $\vec{p} \overset{f(\cdot)}{\longmapsto} H_{opt}$, where $H_{opt}$ is the harm when repairs are sequenced optimally with a single repair crew.
\end{definition}
Let $C$ denote the capital budget available for hardening, $\mathcal{P}$ denotes the repair time after hardening, modeled as a random vector to account for different disaster scenarios, and $y_{lk}$ be a binary variable which is equal to $1$ if hardening strategy $k$ is chosen for edge $l$ and $0$ otherwise. A stochastic optimization model for minimizing the aggregate harm is:
\begin{subequations}
\begin{align}
\underset{\{\Delta p_l\}, \{y_{lk}\}}{\text{min}} \quad &\mathbb{E}[f(\vec{\mathcal{P}})] \label{eqn:hardenidealobj} \\
\text{s.t.} \quad
& \vec{p} - \Delta \vec{p} = \mathbb{E}[\vec{\mathcal{P}}] \label{eqn:hardeningpdistribution} \\
& \sum_{k \in K_l} y_{l k} \leq 1, \;\forall l \in L^D \label{eqn:disjointhardeningactions} \\
& \sum_{l \in L^D} \sum_{k \in K_l} c_{l k} y_{l k} \leq C \label{eqn:hardenbudget} \\
& \Delta p_l = \sum_{k \in K_l} \Delta p_{l k} y_{l k}, \;\forall l \in L^D \label{eqn:dpdef} \\
& y_{lk} \in \{0,1\}, \; \forall l \in L^D, \; \forall k \in K_l \label{eqn:hardenYdec}
\end{align}
\end{subequations}
where the expectations are over all disaster scenarios (the assumption being, different disaster scenarios cause different types/scales of damage and therefore lead to different repair times). While the notation $L^D$ denotes the set of actual damaged edges in the context of the post-disaster scheduling problem (operational phase), we interpret it as the set of all edges which could potentially be damaged in the event of a disaster, the worst case operational scenario, in the context of the hardening problem. Eqn.~(\ref{eqn:hardeningpdistribution}) is the mean-enforcing constraint (which requires that we have knowledge of the first moment of $\vec{\mathcal{P}}$), eqns.~(\ref{eqn:disjointhardeningactions}) and (\ref{eqn:hardenYdec}) force at most one hardening strategy to be chosen per edge from the set $K_l$, eqn.~(\ref{eqn:hardenbudget}) enforces the budget constraint, and eqn.~(\ref{eqn:dpdef}) models the (possible) reduction in repair time of each edge $l$ due to hardening. Observe that the set of constraints (\ref{eqn:disjointhardeningactions}), (\ref{eqn:hardenbudget}) and (\ref{eqn:hardenYdec}) mimics a 0-1 knapsack constraint since we are essentially choosing a subset of hardening strategies from the set of all hardening strategies over all edges, subject to a budget constraint.

Unfortunately, the above stochastic program is difficult to solve, even with perfect knowledge of the distribution of $\vec{\mathcal{P}}$. It is indeed almost impossible to know beforehand the explicit form of $f(\cdot)$ which would result in the optimal harm. Another complicating factor is that the evaluation of the objective function requires knowledge of the distribution function of $\vec{\mathcal{P}}$, while at the same time, this distribution function depends upon the decision variable (eqn.~(\ref{eqn:hardeningpdistribution})). This effectively rules out the applicability of SAA methods. While metaheuristics such as simulated annealing could be used to solve the above problem to (near) optimality, doing so might require an inordinate amount of computation time. We therefore propose a deterministic robust reformulation which is computationally tractable.
\section{The Hardening Problem: Deterministic Robust Reformulation using Jensen's Inequality}
\label{sec:reform}
We begin this section by showing that the min-harm function $f(\cdot)$ is concave. 
In order to develop an intuition for the concavity of $f(\vec{p})$, let us consider an arbitrary network with two possible repair time vectors, $\vec{p_1}$ and $\vec{p_2}$, such that the $l^{th}$ element of $\vec{p_2}$ is smaller than the corresponding element of $\vec{p_1}$ by $\Delta p_l$ while all other elements are equal. That is, $p_2[i] = p_1[i] - \Delta p_l$ if $i = l$ and $p_2[i] = p_1[i]$ if $i \neq l$. Suppose the optimal single crew repair schedule corresponding to $\vec{p_1}$ is $S_1$ and edge $l$ is the $k_1^{th}$ job in $S_1$. Let $R_1$ denote the set of $l$ and all jobs scheduled  after stage $k_1$ in $S_1$. Similarly, let $S_2$ denote the optimal single crew repair schedule corresponding to $\vec{p_2}$ such that edge $l$ is the $k_2^{th}$ job in $S_2$, and $R_2$ denote the set of $l$ and all jobs scheduled  after stage $k_2$ in $S_2$. While $S_1$ is indeed a feasible schedule for $\vec{p_2}$, it may not necessarily be the optimal schedule. In fact, whether the optimal schedule changes under $\vec{p_2}$ depends on $\Delta p_l$, as explained below.

If $\Delta p_l$ is small enough that the optimal schedule under $\vec{p_2}$ is still $S_1$ (i.e., $S_1 = S_2$), it is easy to see that $R_1 = R_2$ and $f(\vec{p_1}) - f(\vec{p_2}) = \Delta p_l \sum_{j \in R_1} w_j > 0$, by positivity of the $w_j$'s. Therefore, any reduction in $p_l$ which does not lead to a change in the optimal schedule will result in a locally linear property for $f(\vec{p})$. 

If $\Delta p_l$ is large enough that $S_1 \neq S_2$, it follows from Algorithm \ref{alg:outree_merge} that edge $l$ cannot be scheduled any later in $S_2$ than in $S_1$; i.e, it must be true that $k_2 \leq k_1$ and $R_2 \supseteq R_1$. Then by the previous paragraph, the slope of the linear section of $\hat{C}_l$ from $p_l - \Delta p_l$ to $p_l$ should be $\sum_{j \in R_2} w_j$. Therefore, in this case, $\sum_{j \in R_2} w_j \geq \sum_{j \in R_1} w_j$, showing a concave property. 
\begin{theorem}
The min-harm function $f(\vec{p})$ is concave.
\label{thm:minharmconcave}
\end{theorem}
\begin{proof}
First, we will show that $f(\vec{p})$ has a piecewise affine structure by extending the analysis above. Using the same arguments, when $\Delta \vec{p}$ is small enough such that the optimal schedule at $\vec{p} - \Delta \vec{p}$ is the same as that at $\vec{p}$, we have: \begin{align}
f(\vec{p}) - f(\vec{p} - \Delta \vec{p}) = \sum_{l \in L^D} \Delta p_l \sum_{j \in R_l} w_j ,
\end{align}
where $R_l$ denotes the set of all jobs scheduled no later than $l$ in the optimal sequence. This shows that the change in the objective value is a linear function of the $\Delta p_l$'s within a small neighborhood of $\vec{p}$, and therefore $f(\vec{p})$ is piecewise affine.

Let $f_{\vec{p_i}}(\vec{p_j})$ denote the harm if the optimal schedule corresponding to $\vec{p_i}$ is used when the actual repair time vector is $\vec{p_j}$. For compactness, we define $f(\vec{p_j}) := f_{\vec{p_j}}(\vec{p_j})$. Since $f(\cdot)$ is affine, for any $\lambda \in (0, 1)$ and $\left(\vec{p_1},\vec{p_2}\right) \in \mathcal{R}^+$ such that $\vec{p}_2 > \vec{p}_1$ and $\vec{p_0} = \lambda \vec{p_1} + (1 - \lambda)\vec{p_2}$, we have:
\begin{align}
f(\vec{p_0}) &= f_{\vec{p_0}}\left(\lambda \vec{p_1} + (1 - \lambda)\vec{p_2}\right) \\ 
&= \lambda f_{\vec{p_0}}(\vec{p_1}) + (1 - \lambda)f_{\vec{p_0}}(\vec{p_2})
\label{eqn:minharmfctlocallinear}
\end{align}
The optimal sequence for $\vec{p_0}$ is not necessarily identical to the optimal sequences for $\vec{p_1}$ and $\vec{p_2}$, and it must be true that $f_{\vec{p_0}}\left(\vec{p_1}\right) \geq f(\vec{p_1})$ and $f_{\vec{p_0}}\left(\vec{p_2}\right) \geq f(\vec{p_2})$. Substituting into eqn~(\ref{eqn:minharmfctlocallinear}),
\begin{align}
f(\vec{p_0}) &= \lambda f_{\vec{p_0}} (\vec{p_1}) + (1 - \lambda) f_{\vec{p_0}}(\vec{p_2}) \\
&\geq \lambda f(\vec{p_1}) + (1 - \lambda) f(\vec{p_2}),
\label{eqn:2}
\end{align}
which proves that $f(\cdot)$ is a concave function.
\end{proof}
Since $f(\cdot)$ is concave, Jensen's inequality~\cite{jensen1906fonctions} holds and the objective function~(\ref{eqn:hardenidealobj}) can be naturally upper bounded as follows:
\begin{equation}
E[f(\vec{\mathcal{P}})] \leq f(E[\vec{\mathcal{P}}])
\label{eqn:jensen}
\end{equation}
Fig.~\ref{fig:comparisonfeef} provides an illustration of the above inequality when $\vec{\mathcal{P}}$ follows an univariate geometric distribution. 

The preceding discussion motivates the following deterministic robust reformulation (note that constraint~(\ref{eqn:hardeningpdistribution}) has been wrapped into the objective function): 
\begin{align}
\underset{\{\Delta p_l\}, \{y_{lk}\}}{\text{min}} \quad &f(\vec{p} - \Delta \vec{p}) \label{eqn:hardenrefobj} \\
\text{s.t.} \quad
& (\ref{eqn:disjointhardeningactions}) \sim (\ref{eqn:hardenYdec}) \nonumber
\end{align}
As will be apparent from the next section, the above model is a key development which allows for an integrated treatment of the restoration process and the hardening problem.

We conclude this section with a note on the worst case impact on the objective function caused by the upper bounding by Jensen's inequality. Assume that the support of $\vec{\mathcal{P}}$ is bounded, i.e., $\vec{\mathcal{P}} \in [0, \vec{p}_{max}]$. Then, it follows from Theorem 1 in~\cite{simic2008global} that:
\begin{align}
f\left(\mathbb{E}[\vec{\mathcal{P}}]\right) - \mathbb{E}\left[f(\vec{\mathcal{P}})\right] \leq f\left(\vec{p}_{max}\right) - 2 f\left(\frac{\vec{p}_{max}}{2}\right)
\end{align}
\begin{figure}[htbp]
\centering
\includegraphics[trim = 0 10 0 50, clip, width = 0.5\columnwidth]{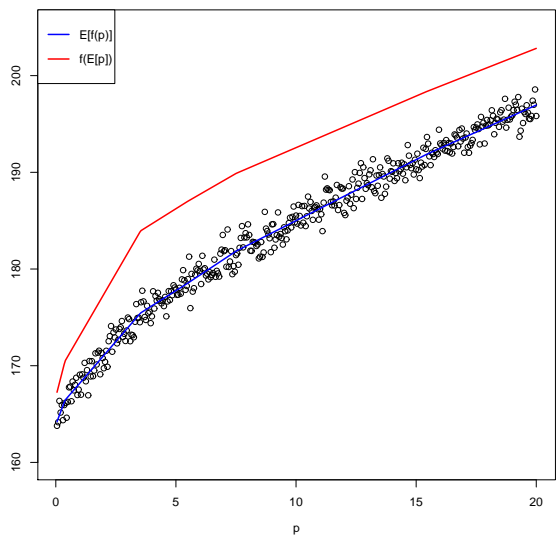}
\caption{Comparison of one-dimensional $f(\mathbb{E}[\mathcal{P}])$ and $\mathbb{E}[f(\mathcal{P})]$ when $\mathcal{P}$ follows a univariate geometric distribution.}
\label{fig:comparisonfeef}
\end{figure}
\section{Restoration Process Aware Hardening Problem}
\label{sec:solution}
Usually, the restoration problem and the hardening problem are treated separately because the former is an operational problem while the latter is a planning problem. However, we argue that the two problems should not be treated in isolation because hardening can affect the repair times, which in turn, can influence the restoration times through the sequencing process and thereby the aggregate harm or resilience. The model that we formulate is similar to single machine scheduling with controllable processing times, which dates back to the 1980s~\cite{nowicki1990survey}. See Section~2 in~\cite{shioura2016application} for a review of recent advances. Our problem is more complicated in the sense that the effect of hardening decisions (or costs of compression amount in that setting) are not just linear, instead they are embedded in the sequencing problem. 

In this section, we discuss two solution approaches for the so-called `restoration process aware hardening problem' (RPAHP), first an MILP formulation, followed by a heuristic algorithm framework inspired by a continuous convex relaxation. Although we consider only one repair crew, we emphasize that this does not preclude the deployment of multiple crews for post-disaster restoration. Since it is impossible to know the availability of the number of repair crews in the event of a disaster at the network planning stage, we choose to build our integrated model assuming one repair crew. The key point being, a network which has been designed/hardened with an eye on the restoration process (albeit, with one repair crew) will be much quicker to restore post-disaster when additional repair crews might be available, as opposed to a network which has been designed/hardened with no consideration given to the restoration process. Furthermore, as shown in~\cite{tan2017scheduling}, edges which are scheduled earlier in the single crew repair sequence should be repaired with a higher priority in the multi-crew schedule. Therefore, consideration of a single repair crew should provide an indicative result, irrespective of the number of repair crews.
\subsection{MILP Formulation}
\label{subsec:hardenMILP}
In Section~\ref{sec_sequencing_MILP}, we developed an MILP model for optimizing the repair schedule with one repair crew, while in Section~\ref{sec:reform}, we developed a deterministic MILP reformulation of the hardening problem, both with the same objective, minimization of the aggregate harm. These two models can be easily incorporated into an integrated MILP formulation, as shown below $(\vec{x} := \{x_l^t, t = [1,T], \, l \in L^D\}$, $\vec{u} := \{u_i^t, t = [1,T], \, i \in N\}$, $\Delta \vec{p} := \{\Delta p_l, l \in L^D\}$, $\vec{y} := \{y_{lk}, l \in L^D, \, k \in K_l\})$: 
\begin{align}
\underset{\vec{x}, \vec{u}, \Delta \vec{p}, \vec{y}}{\text{min}} \quad &\sum_{t=1}^T h^t  \\
\text{s.t} \quad 
& (\ref{eqn:uinit}) \sim (\ref{eqn:dinit}) \nonumber \\
& d^t \geq d^{t - 1} + (p_l - \Delta p_l) \times x_l^t, \, \forall t \in [1,T], \, \forall l \in L^D\label{eqn:hardendefd} \\
& (\ref{eqn:defh}) \sim (\ref{eqn:linearizeh}) \nonumber \\
& (\ref{eqn:disjointhardeningactions}) \sim (\ref{eqn:hardenYdec}) \nonumber
\end{align}
Observe that the impact of hardening, $\Delta p_l$,  is incorporated into constraint~(\ref{eqn:hardendefd}). The product of $\Delta p_l$ and $x_l^t$ on the r.h.s of eqn.~(\ref{eqn:hardendefd}) can be easily linearized using the big-$M$ method, details of which are omitted.
\subsection{A continuous convex relaxation}
\label{subsec:relaxation}
As stated previously, $f(\vec{p})$ and $f(\vec{p} - \vec{\Delta p})$ are both concave piecewise affine functions in $\vec{p}$. In general, concave minimization problems are $\mathcal{NP}$-hard~\cite{garey1976some}. In our case, there are at most $n!$ affine pieces, corresponding to $n!$ number of affine possible sequences, where $n= |L^D|$ is the number of damaged edges. The region corresponding to each affine piece is a cone, and all cones share a common  vertex.

The RPAHP involves two types of decision variables, the sequencing variables (the $x$'s and $u$'s) and the hardening variables (the $\Delta p$'s). Given the hardening variables, it is straightforward to see that the joint optimization problem reduces to the single crew sequencing problem, which can be solved optimally in polynomial time as stated previously in Section~\ref{subsec:seqheur}. 

Now let us consider the case where the sequencing variables are fixed. Let:
\begin{align}
\Omega_l := \sum_{j \in R_l} w_j ,
\label{eqn:defOmega}
\end{align}
where $R_l$ is the set of some edges $l \in L^D$ and all its successors in the given sequence. As discussed in Section~\ref{sec:reform}, the quantity $\Omega_l$ represents the reduction in aggregate harm per unit decrease in $p_l$. The objective function for the hardening problem can now be recast as $f(\vec{p}) = \sum_{l \in L^D} \Omega_l \, p_l$, which implies:
\begin{align}
f(\vec{p} - \Delta \vec{p}) = \sum_{l \in L^D} \Omega_l (p_l - \Delta p_l) = \sum_{l \in L^D} \Omega_l p_l - \sum_{l \in L^D} \Omega_l \Delta p_l
\label{eqn_knapheur_preobj}
\end{align}
Since the first term on the extreme r.h.s of eqn.~(\ref{eqn_knapheur_preobj}) is a constant,  instead of minimizing $f(\vec{p} - \Delta \vec{p})$, an equivalent formulation is:
\begin{subequations}
\label{eqn:mckp}
\begin{align}
\underset{\vec{y}}{\text{max}} \quad &\sum_{l \in L^D}  \sum_{k \in K_l} \Omega_l \Delta p_{lk} y_{lk} \\
\text{s.t.} \quad
& \sum_{k \in K_l} y_{lk} \leq 1\,\;, \forall l \in L^D \\
& \sum_{l \in L^D}  \sum_{k \in K_l} c_{lk} y_{lk} \leq C \\
& y_{lk} \in \{ 0, 1 \},\,\; \forall l \in L^D,\,\; \forall k \in K_l
\end{align}
\end{subequations}
This model is similar to that of the multiple choice knapsack problem~\cite{sinha1979multiple}, where $\Omega_l \Delta p_{lk}$'s are the value coefficients and $c_{lk}$'s are the cost coefficients. Since the multiple choice knapsack is known to be $\mathcal{NP}$-hard, we propose an algorithm based on convex envelopes and LP relaxation, similar to~\cite{kameshwaran2009nonconvex}.
\begin{definition}[Convex envelope~\cite{horst2013global}]
Let $M \subset \mathcal{R}^n$ be convex and compact and let $g : M \rightarrow \mathcal{R}$ be lower continuous on $M$. A function $\hat{g} : M \rightarrow \mathcal{R}$ is called the convex envelope of $f$ on $M$ if it satisfies:
\begin{itemize}
\item $\hat{g}(x)$ is convex on $M$,
\item $\hat{g}(x) \leq g(x)$ for all $x \in M$,
\item there is no function $h : M \rightarrow \mathcal{R}$ satisfying (1), (2) and $g(x_0) < h(x_0)$ for some point $x_0 \in M$.
\end{itemize}
\end{definition}
Intuitively, the convex envelope is the best underestimating convex function of the original function. Details of a polynomial time algorithm for computing the convex envelope of a piecewise linear function can be found in~\cite{kameshwaran2009nonconvex}. 

Given a discrete function of $c_{lk}$ vs. $\Delta p_{lk}$ for some edge $l$ and a set of all hardening actions $k \in K_l$, we first connect the neighboring points, starting from the origin,  to construct a continuous piecewise linear cost function $C_l(\Delta p_l)$, where $\Delta p_l$ is the relaxed continuous decision variable. It follows from  Assumption~\ref{asmp:mono} that $C_l$ is a strictly increasing function. Let $\hat{C}_l$ denote the convex envelope of $C_l$ and $\hat{K}_l = \{1, 2, \cdots, \lvert \hat{K}_l \rvert\}$ denote the set of breakpoints/knots on the convex envelope (excluding the origin) corresponding to the hardening strategies in consideration, indexed in ascending order of $\Delta p_{lk}$. The linear relaxation of~(\ref{eqn:mckp}), based on the convex envelope approximations, can then be formulated as:
\begin{subequations}
\label{eqn:LPrelax}
\begin{align}
\underset{\Delta \vec{p}}{\text{max}} \quad &\sum_{l \in L^D}  \Omega_l \Delta p_{l} \label{eqn:relaxhardenrefobj} \\
\text{s.t.} \quad
& Q_{l} \geq \underset{k \in \hat{K}_l}{\mbox{max}}  \left[ \mu_{lk} \left(\Delta p_{l} - \alpha_{lk}\right) + b_{lk} \right], \, \forall l \in L^D \label{eqn:hardenconst1_AltLP} \\
& \sum_{l \in L^D} Q_{l} \leq C \label{eqn:hardenconst2_AltLP}  \\
& 0 \leq \Delta p_l \leq \Delta p_{l,  \lvert \hat{K}_l \rvert}, \, \forall l \in L^D
\end{align}
\end{subequations}
where $\mu_{lk}$ and $b_{lk}$ are the slope and intercept (see below) of the $k^{th}$ piece of $\hat{C}_l$, $\alpha_{lk}$ is the lower breakpoint of the $k^{th}$ piece of $\hat{C}_l$, and $Q_l$ is an intermediate decision variable which accounts for the budget spent on edge $l$. This formulation is similar to the conventional continuous knapsack problem, and it turns out that the optimal values of $\Delta p_l$ are always from the set $\{0, \mbox{some } \beta_{lk}, (\alpha_{lk}, \beta_{lk})\}$, where $\beta_{lk}$ is the upper breakpoint of the $k^{th}$ piece of $\hat{C}_l$. Furthermore, at most one $\Delta p_l$ can have an intermediate value in the range $(\alpha_{lk}, \beta_{lk})$ in the optimal solution. For some $l$ and $k > 1$, the intercept parameter, $b_{lk}$, is: $b_{lk} = \hat{C}_l\left(\alpha_{lk}\right) = \hat{C}_l\left(\beta_{l(k-1)}\right)$. For $k = 1$, $\alpha_{lk} = \hat{C}_l\left(\alpha_{lk}\right) = b_{lk} = 0$.
%

The preceding LP relaxation~(\ref{eqn:LPrelax}) can also be solved optimally using a greedy algorithm by first sorting the ratios $\left\{\frac{\Omega_l}{\mu_{lk}}\right\}$ in a descending order, and then choosing the components (and the degree of hardening) based on that sorted list iteratively, until the budget is exhausted. Ties, if any, during the selection process, are broken arbitrarily. We use a $\Delta p$ variable for each edge $l$ and each segment $k$ of $\hat{C}_l$. All these $\Delta p_{lk}$ variables are initialized to $0$. Once an $(l,k)$ selection is made from the sorted list at any iteration $T$, say $l = l_T$ and $k = k_T$, we set $\Delta p_{l_T k_T}$ equal to the maximum value possible within the range $[\alpha_{l_T k_T}, \beta_{l_T k_T}]$ such that the `cumulative budget' at the end of iteration  $T$ does not exceed $C$. Typically, this maximum value will be at the upper breakpoint $\beta_{l_T k_T}$, unless, doing so results in a budget violation. In that case, a proper value within the range $(\alpha_{l_T k_T}, \beta_{l_T k_T})$ is chosen such that the budget is met exactly. During implementation, we first assign $\Delta p_{l_T k_T} \leftarrow \beta_{l_T k_T}$ and then check the budget violation criterion (explained below). If the criterion is violated, $\Delta p_{l_T k_T}$ is reassigned a proper value such that the budget is met exactly. It is obvious that this reassignment needs to be done at most once during the operation of the algorithm. 
At the end of every iteration, we evaluate the expression: 
\begin{align} 
  \Lambda &= \sum_{l \in L^D} \, \hat{C}_l \left(\underset{k \in \hat{K}_l}{\mbox{max}} \left\{ \Delta p_{lk} \right\} \right) , 
 \label{greedyAlgoTerminationCodn}
\end{align}
which represents the cumulative budget consumed till the current iteration. The algorithm terminates when $\Lambda = C$. Upon termination, the optimal $\Delta p_l$ values can be obtained from the $\Delta p_{lk}$ values as follows:
\begin{align}  
  \Delta p_l &= \underset{k \in \hat{K}_l}{\mbox{max}} \left\{ \Delta p_{lk} \right\}.
 \label{greedyAlgoVariableRecovery}
\end{align}
%
We now provide an example which helps illustrate the operation of the algorithm. 

Consider a scenario where two edges are to be repaired, $l=1,2$, and the hardening cost functions for the two edges are as shown in Fig.~\ref{fig:illusgreedy}. Suppose $C=10$ and $\Omega_1 = \Omega_2 = 1$. 

\begin{figure}[htbp]
\centering
\includegraphics[width = 1.0\columnwidth]{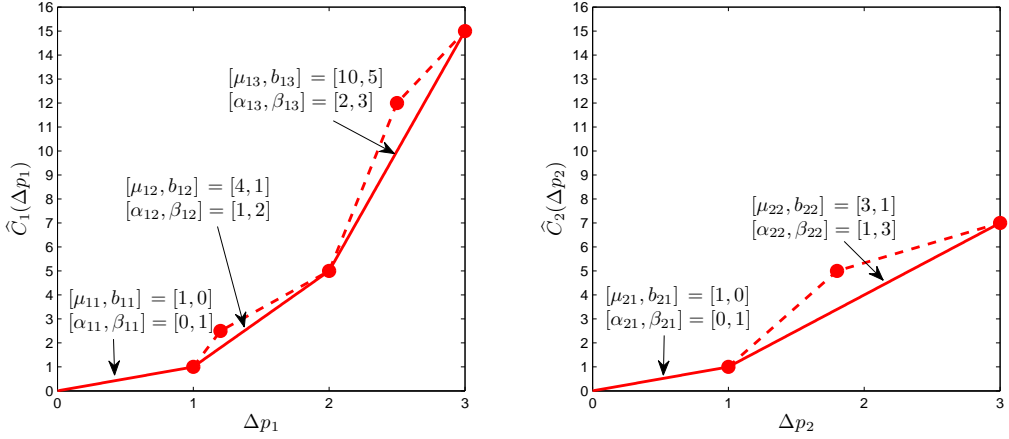}
\caption{Hardening cost functions for illustrating the greedy algorithm used to solve the continupus knapsack-like problem. The solid circles represent the actual discrete hardening strategies and costs, the dashed lines represent the piecewise linear constructions, while the solid lines represent the convex envelope approximations.}
\label{fig:illusgreedy}
\end{figure}

The selections made by the greedy algorithm at each step are as follows: 
\begin{itemize}
  \item \textit{Step 1}: Breaking ties arbitrarily, choose $(l=2, k=1)$, set $\Delta p_{21} = 1$, cumulative hardening cost $= C_1(1) = 1$.  
  \item \textit{Step 2}:  Choose $(l=1,k=1)$, set $\Delta p_{11} = 1$, cumulative hardening cost $= C_1(1) + C_2(1) = 1+1=2$. 
  \item \textit{Step 3}:  Choose $(l=2, k=2)$, set $\Delta p_{22} = 3$, cumulative hardening cost $ = C_1(1) + C_2(\mbox{max}[1,3]) = 1 + 7 = 8$.  
  \item \textit{Step 4}:  Choose $(l=1, k=2)$, set $\Delta p_{12} = 1.5$, cumulative hardening cost $= C_1(\mbox{max}[1,1.5]) + C_2(\mbox{max}[1,3]) = 3 + 7 = 10$. Note that, unlike the previous $3$ steps, we can only afford $1.5$ units of  hardening corresponding to $(l=1, k=2)$ so that the budget is not violated.
\end{itemize}
The LP solutions are therefore the points $(1.5,3)$ and $(3,7)$ for edges $1$ and $2$ respectively.
\subsection{An iterative heuristic algorithm}
\label{subsec:hardenGreedy}
We now discuss an iterative heuristic algorithm for solving the RPAHP. First, we note that the solutions obtained from the greedy algorithm used to solve the convex relaxation formulation~(\ref{eqn:LPrelax}) may need to be \emph{rounded down} to the nearest lower breakpoints on the convex envelopes so that the hardening strategy is feasible for each edge. In the context of the above example, we would therefore select the point $(1, 1)$ for edge $1$ (left panel of Fig.~\ref{fig:illusgreedy}). No rounding is necessary for edge $2$ since the point selected by the greedy algorithm, $(3,7)$, does correspond to an actual hardening strategy. By rounding down, whenever necessary, we ensure that the budget constraint will not be violated. 

However, after completion of the rounding process, we may find that a portion of the budget has been left unspent. We therefore incorporate a \emph{backfill} heuristic which iteratively solves LP relaxations of the form~(\ref{eqn:LPrelax}) with the unspent budget from the previous iteration and the remaining available hardening options, along with updated convex envelopes and the optimal repair sequence, followed by a rounding down to a feasible hardening strategy. In the context of the example provided in the previous sub-section, rounding down the LP solution for edge $1$ to the point $(1,1)$ creates an unspent budget of $2$ units, which becomes the new budget for the second iteration. During the second iteration, the points $(0,0)$ (no hardening is a feasible option in iteration $1$), $(2,5)$, $(2.5,12)$ and $(3,15)$ in the left panel of Fig.~\ref{fig:illusgreedy} are no longer in consideration and the convex envelope is recomputed over the set of points $(1,1)$ and $(1.2,2.5)$, with the former being the new origin. Since the optimal repair schedule depends on the repair times, we update the schedule after every iteration $t$ with the new repair time vector, $\vec{p}(t+1) \leftarrow \vec{p}(t) - \Delta \vec{p}(t)$. The backfill process terminates whenever the budget has been spent exactly, or, when no further enhancement is possible on any edge without exceeding the budget. 

Summarizing what we have so far, we now describe a general framework of a multi-run heuristic algorithm for solving the RPAHP, as shown in Fig.~\ref{fig:flowchart}. Broadly speaking, the approach involves three major stages. In the first stage, we compute the single crew optimal sequence, given $\vec{p}$, the expected repair time vector before hardening. In the second stage, we use the optimal repair sequence obtained from the first stage and solve the LP relaxation~(\ref{eqn:LPrelax}) using the convex envelopes of the hardening cost functions, followed by rounding, which yields a set of feasible hardening decisions. In the third stage, we implement a backfill procedure by re-solving the LP relaxation~(\ref{eqn:LPrelax}) with updated information, as described in the previous paragraph. We provide three options in Fig.~\ref{fig:flowchart} which differ in how often the repair sequence is updated based on some hardening decisions. Option $3$, which is the most aggressive, updates the repair sequence \emph{after every iteration of the greedy algorithm used for solving the LP relaxation}~(\ref{eqn:LPrelax}). To avoid clutter, we have opted to show the feedback arrow in Option $3$ going directly to the `blue greedy algorithm box', instead of expanding the details of it. Option $1$, which is the most conservative, does not update the repair sequence at all and uses the initial $\Omega_l$'s until termination. Option $2$ represents a middle ground and updates the repair sequence \emph{after completion of the greedy algorithm used for solving the LP relaxation}~(\ref{eqn:LPrelax}). Implementation details of these three options are shown in Algorithm~\ref{alg:multirungreedy}.

%

\begin{figure}[H]
\centering
\includegraphics[trim = 80 250 90 210, clip, width = 0.80\columnwidth]{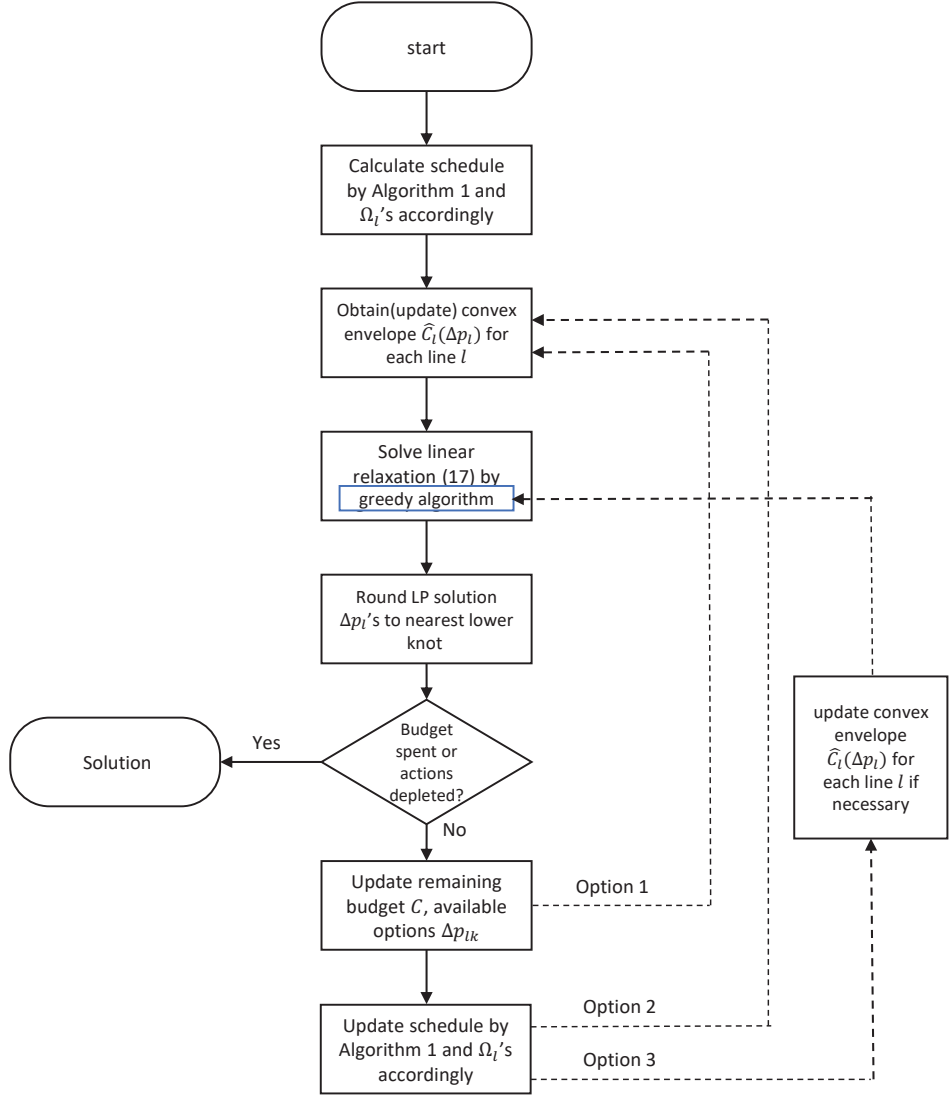}
\caption{Flowchart of the multi-run algorithm framework for solving the RPAHP. Note that updates of the convex envelope in Option $3$ are  necessary only once per backfill step.}
\label{fig:flowchart}
\end{figure}
\begin{algorithm}
\caption{Algorithms for restoration process aware distribution systems hardening.}
\label{alg:multirungreedy} 
\begin{algorithmic}[1]
	\STATE Compute the optimal sequence given the expected repair time $\vec{p}$, using Algorithm~\ref{alg:outree_merge};
    \STATE Calculate the weights $\Omega_l$ according to eqn.~(\ref{eqn:defOmega});
    \STATE Obtain the convex envelopes of costs $\hat{C}_l \left(\Delta p_{l} \right)$ with $\hat{K}_l$ pieces, along with the coefficients $\mu_{lk}, b_{lk}, \alpha_{lk}$ and $\beta_{lk}$, for each edge $l \in  L^D$;
    \STATE $H \leftarrow \emptyset$; 
    \STATE $k_l \leftarrow 1, \;\, \forall l \in L^D$;
    \WHILE{true}
    \STATE find $l \in L^D \setminus H$ with largest $\frac{\Omega_l}{\mu_{l, k_l}}$;
    \STATE let $\Delta p_l = \beta_{lk_l}$ and calculate the current cost $\Lambda = \sum_{l \in L^D} \, \hat{C}_l \left(\Delta p_{l} \right)$;
    \IF {$\Lambda = C$}
    \STATE break;
    \ELSIF {$\Lambda > C$}
    \STATE $\Delta p_l = \beta_{l,k_l-1}$;
    \STATE \emph{Option 2 \& 3: Update the optimal sequence given the current expected repair time $\vec{p} - \Delta \vec{p}$} and then update $\Omega$'s.
    \STATE Update the convex envelope of cost $\hat{C}_l \left(\Delta p_{l} \right)$ for edge $l$ and then update the coefficients $\mu_{lk}, b_{lk}, \alpha_{lk}$ and $\beta_{lk}$, for each edge $l \in  L^D$;
    \ELSIF {$k_l = \lvert \hat{K}_l \rvert$}
    \STATE $\Delta p_l = \beta_{l,k_l-1}$;
    \STATE $H \leftarrow \{ H, l \}$;
    \ELSE
    \STATE $k_l = k_l + 1$;
    \STATE \emph{Option 3: Update the optimal sequence given the current expected repair time $\vec{p} - \Delta \vec{p}$} and then update $\Omega$'s.
    \STATE continue;
    \ENDIF
    \IF {$\lvert H \rvert = \lvert L^D \rvert$}
    \STATE break;
    \ENDIF
    \ENDWHILE
\end{algorithmic}
\end{algorithm}
%
%
%
\section{Case studies}
\label{sec:case}
\subsection{IEEE 13 node test feeder}
%
We first test the MILP and heuristic approaches discussed in the previous section on the IEEE $13$ node test feeder with randomly generated $C_l$'s and two different budgets. Values of $\mathbb{E}[f(\cdot])$ in this section were computed using Monte Carlo simulations assuming an independent geometric distribution for each $p_l$. With a budget of $C=5$, hardening actions did not result in different repair schedules and both the MILP and heuristic approaches (all three options) yielded identical results, as shown in~Table~\ref{tab:hardenresult5}. With a budget of $C = 8$, even though the hardening actions suggested by the MILP and heurisitic approaches (all three options) differ for two edges, as shown in Table~\ref{tab:hardenresult8}, the objective values obtained from the greedy algorithm, both for $\mathbb{E}[f(\cdot)]$ and its upper bound $f(\mathbb{E}[\cdot])$, are very close to those provided by the MILP formulation. In fact, the ratio of the $f(\mathbb{E}[\cdot])$ measure from the greedy algorithm to the $\mathbb{E}[f(\cdot)]$ measure from the MILP algorithm is approximately $1.03$ for $C=5$ and $1.04$ for $C=8$ (note that this ratio captures the worst case performance loss, including the effect of upper bounding the true objective function using Jensen's inequality). 
\begin{table}[h]
\centering
\begin{tabular}{c|cc}
edge $l$ & $\Delta p_l$ by MILP & $\Delta p_l$ by Greedy Algorithms \\ \hline
671-680 & 0.6    & 0.6    \\
650-632 & 0.6  & 0.6 \\
671-684 & 0.4  & 0.4 \\
645-646 & 0.4 & 0.4 \\
684-652 & 0.2    & 0.2 \\
632-645 & 0.8 & 0.8 \\
632-633 & 0.8  & 0.8 \\
633-634 & 0    & 0 \\
632-671 & 0.4  & 0.4 \\
671-692 & 0    & 0 \\
692-675 & 0    & 0 \\
684-611 & 0.2  & 0.2 \\  \hline
$f(\mathbb{E}[\cdot])$ & 15.368 & \boxed{15.368} \\
$\mathbb{E}[f(\cdot)]$ & \boxed{14.950} & 14.950
\end{tabular}
\caption{Comparison of hardening results on the IEEE $13$ node test feeder with a budget of $C = 5$.} 
\label{tab:hardenresult5}
\end{table}
\begin{table}[h],
\centering
\begin{tabular}{c|cc}
edge $l$ & $\Delta p_l$ by MILP & $\Delta p_l$ by Greedy Algorithms \\ \hline
671-680 & 0.6    & 0.6    \\
650-632 & 0.6  & 0.6 \\
671-684 & 0.2  & 0.4 \\
645-646 & 0.4 & 0.4 \\
684-652 & 0.2    & 0.2 \\
632-645 & 1.0 & 0.8 \\
632-633 & 0.8  & 0.8 \\
633-634 & 0    & 0.5 \\
632-671 & 1.4  & 1.4 \\
671-692 & 0    & 0.2 \\
692-675 & 0    & 0 \\
684-611 & 0.2  & 0.4 \\  \hline
$f(\mathbb{E}[\cdot])$ & 14.876 & \boxed{14.917} \\
$\mathbb{E}[f(\cdot)]$ & \boxed{14.300} & 14.498
\end{tabular}
\caption{Comparison of hardening results on the IEEE $13$ node test feeder with a budget of $C=8$.}
\label{tab:hardenresult8}
\end{table}

Next, we varied the hardening budget from $0$ to $20$. These results are summarized in Fig.~\ref{fig:relation}. In each case, all three options within the heuristic framework produced identical solutions. The MILP and heuristic approaches yielded almost identical results when using the $f(\mathbb{E}[\cdot])$ measure so that their plots almost overlap. The plots corresponding to the $\mathbb{E}[f(\cdot)]$ measure are also very close, considering the errors introduced by Monte Carlo simulations. Expectedly, the aggregate harm decreases (resilience increases) as the hardening budget increases, but the returns indicate a diminishing trend. 
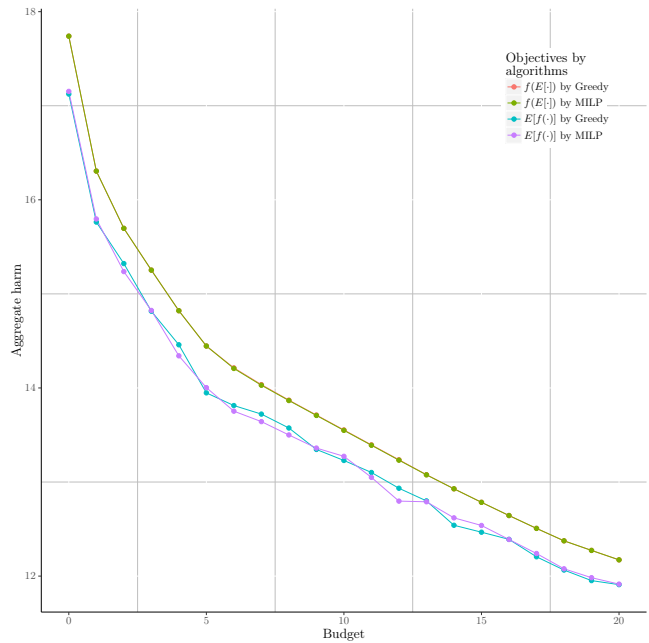
\begin{figure}[H]
\centering
\scalebox{0.42}{\input{budget_versus_harmMILP_harmGreedy_objMILP_objGreedy.tex}}
\caption{Comparison of hardening results on the IEEE $13$ node test feeder with a varying budget.}
\label{fig:relation}
\end{figure}
\subsection{IEEE 37 node test feeder}
Next, we ran our algorithms on one instance of the IEEE $37$ node test feeder~\cite{kersting2001radial}. Since the running time of the MILP formulation increases exponentially with network size, we allocated a time budget of ten hours. In contrast, all three heuristic options yielded a solution within seconds. Table~\ref{tab:37hardenresult} shows the edges for which the MILP and heuristic  approaches produced different hardening results.

%
\begin{table}[htbp]
\centering
\begin{tabular}{c|cccc}
edge $l$               & MILP(10 hours) & Option 1 & Option 2 & Option 3 \\ \hline
(744, 729))            & 0.8            & 0        & 0        & 0        \\
(702, 703)             & 0.2            & 0.2      & 0.2      & 0.3      \\
(708, 733)             & 0              & 0.3      & 0.3      & 0.3      \\
(702, 705)             & 2.1            & 0.4      & 0.4      & 2.1      \\
(734, 737)             & 0              & 0        & 0.2      & 0.2      \\
(708, 732)             & 0              & 0.2      & 0.2      & 0        \\
(734, 710)             & 0.1            & 1.4      & 1.7      & 0.1      \\ \hline
$f(\mathbb{E}[\cdot])$ & 843.08         & 842.04   & 843.84   & 837.93   \\
$\mathbb{E}[f(\cdot)]$ & 672.21         & 667.35   & 667.42   & 666.09                            
\end{tabular}
\caption{Comparison of reduction in repair times, $\Delta p_l$'s, due to hardening on the IEEE $37$ node test feeder with a budget of $C = 200$.}
\label{tab:37hardenresult}
\end{table}
In order to compare the performances of the three options within the heuristic framework, we then varied the hardening budget from $1$ to $400$. Fig.~\ref{fig:compoptions} summarizes these results. Intuitively, when the hardening budget is small, we expect the three options to behave similarly since reductions in repair times, if any, are likely to be small enough so as not to trigger a change in the repair schedule, rendering the `update schedule' step in Fig.~\ref{fig:flowchart} moot. Similarly, when the hardening budget is large, all three options should behave similarly since most edges are likely to be hardened to the maximum degree possible at the end of the first run, and in this case, the `update schedule' step would be inconsequential since the algorithm would tend to terminate after the first run. As can be observed from Fig.~\ref{fig:compoptions}, the three options indeed behave similarly at either end of the budget spectrum, but produce somewhat different results for intermediate budgets (in the range $21-303$), although the differences are not appreciable. For a better understanding of the average performance of the three options, we conducted $200$ trials with randomly generated hardening cost functions and a budget of $C=282$. Option $1$ turned out to be the best on $53$ trials, option 2 on $69$ trials, and  option 3 on $158$ trials. Note that the numbers do not add up to $200$ since ties were counted while ranking the three options. All three options produced identical results on $19$ trials. However, the largest difference that we observed between any two options was $4.7\%$. Consequently, we recommend Option $1$ as the preferred option if ease of implementation and fastest computational performance are desired.

\begin{figure}[H]
\centering
\includegraphics[width = 0.5\columnwidth]{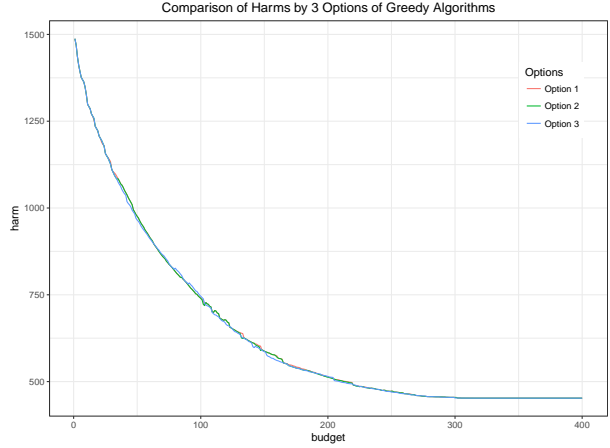}
\caption{Comparison of harms by $3$ options of the heuristic framework on the IEEE $37$ node test feeder.}
\label{fig:compoptions}
\end{figure}
\subsection{IEEE 8500 node test feeder}
Finally, we tested the performance of the heuristic algorithm (option $1$ only) on one instance of the IEEE $8500$ node test feeder medium voltage subsystem \cite{arritt2010the} containing roughly $2500$ edges. We did not even attempt to solve the ILP model in this case, but the heuristic algorithm took just $9.36$ secs. to solve this instance. Results are shown in Fig.~\ref{fig:8500budget}.
\begin{figure}[H]
\centering
\scalebox{0.4}{\input{greedyharm8500.tex}}
\caption{Illustrating the trend of diminishing returns with an increasing budget on the IEEE $8500$ node test feeder.}
\label{fig:8500budget}
\end{figure}
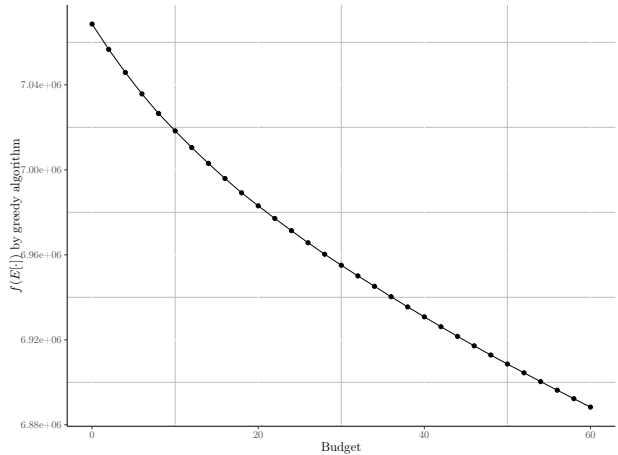
\section{Conclusions}
In this paper, we investigated the problem of optimally hardening a distribution network to be resilient to natural disasters. Motivated by recent work on resilient infrastructure systems in civil engineering, we proposed an equivalent definition of resilience with a clear physical interpretation. This allowed us to integrate the post disaster restoration process and the planning stage component hardening decision process into one problem, which, we argued, is necessary since both aspects ultimately contribute to system resilience. This is a major departure from most current research where the two aspects of resilience are treated separately. We first modeled the restoration problem as an MILP and the hardening problem as a stochastic program, which was then reformulated as a robust deterministic model using Jensen's inequality for the sake of computational tractability. Finally, we unified the sequencing and hardening aspects and proposed an integrated MILP model as well as a multi-run greedy algorithmic framework. The expected component repair times, which are updated during the algorithm when necessary, are used to generate an optimal single crew repair sequence, based on which hardening decisions are made sequentially in a greedy manner. Simulations on IEEE standard test feeders show that the heuristic approach provides near-optimal solutions efficiently even for large networks. 

\bibliography{hardening_ref}

\appendix
\section{Proof of optimality of the greedy algorithm for solving the LP relaxation~(\ref{eqn:LPrelax})}
\label{appdx:proofgreedyoptimality}
\begin{proof}
Suppose the greedy algorithm terminates at some iteration $T$ and that it is not optimal. In that case, there must be some edge $l$ and some piece of the hardening cost function $k$ which we could have chosen at some subsequent iteration, say $t = T^+$, such that, if this was included in the optimal solution, it would improve the objective value without violating the budget constraint. Denote the $l$ and $k$ for this edge by $l = l_{T^+}$ and $k = k_{T^+}$.

First, let $l_{T^+}$ be such that $l_{T^+} \in \mathcal{R}_T$, where $\mathcal{R}_T$ denotes the set of edges chosen for hardening till iteration $T$, i.e., at least one of the $\left(\Delta p_{lk}: \forall l \in \mathcal{R}_T\right)$'s is greater than $0$. In essence, what we are saying is that, $l_{T^+}$ is an edge which was already chosen for hardening until iteration $T$, but, we could actually have afforded a greater  degree of hardening, $k_{T^+}$, compared to where the algorithm terminated. Recall that the greedy algorithm chooses $l$ and $k$ sequentially based on the ratios $\left\{ \frac{\Omega_l}{\mu_{lk}} \right\}$, sorted in a descending order, and therefore, for some $l$, smaller $k$'s are chosen first. If this were possible, it would of course improve the objective value, but the budget constraint would be violated since the algorithm terminated with the budget being met exactly and hardening the edge $l_{T^+}$ to a greater degree $k_{T^+}$ than what was chosen till iteration $T$ would necessarily drive up the hardening cost for that edge (this follows from convexity of the $\hat{C}_l$'s). This proves that it is not feasible to increase the degree of hardening of an edge which has already been chosen for some degree of hardening by the greedy algorithm until termination, without violating the budget constraint.

Next, let $l_{T^+}$ be such that $l_{T^+} \not\in \mathcal{R}_T$. In essence, what we are saying is that, $l_{T^+}$ is not an edge which was already chosen for hardening until iteration $T$, but, including this edge in the optimal solution, possibly at the expense of one or more edges (this has to be the case since adding this edge for hardening without taking out other edges would obviously violate the budget constraint) which were chosen for hardening till iteration $T$, could improve the objective value without violating the budget constraint. Let us assume that this `addition-and-subtraction' strategy does not violate the budget constraint. We now show that this strategy cannot lead to an increase in the objective value. Let $\mathcal{Q}_T \subseteq \mathcal{R}_T$ denote some subset of edges chosen for hardening till iteration $T$, which we could replace with $l_{T^+}$. For the addition-and-subtraction strategy to be better, we need to show that:
\begin{align}
  \Omega_{l_{T^+}} \Delta p_{l_{T^+}} &> \sum_{l \in \mathcal{Q}_T} \Omega_{l} \Delta p_{l}
\label{greedyProofEqn0}
\end{align}
However, we argue below that the addition-and-subtraction strategy is not better, and therefore what holds is that:
\begin{align}
 \Omega_{l_{T^+}} \Delta p_{l_{T^+}} &\leq \sum_{l \in \mathcal{Q}_T} \Omega_{l} \Delta p_{l}
\label{greedyProofEqn1}
\end{align}
Without any loss of generality, we assume that iteration $T^+$ is the first time after $T$ that edge $l_{T^+}$ has been chosen. If this is not the case, the arguments below apply to that iteration, after $T$, when $l_{T^+}$ was first chosen. Since the hardening selection process for any edge kicks off with $k=1$, it follows that $\Delta p_{l_{T^+}} = \Delta p_{l_{T^+}(k=1)}$, where $\Delta p_{l_{T^+}(k=1)} \in \left[ \alpha_{l_{T^+}(k=1)}, \beta_{l_{T^+}(k=1)}\right]$. Let us assume that $\Delta p_{l_{T^+}(k=1)}$ can be set to its maximum possible value, $\beta_{l_{T^+}(k=1)}$, without violating the budget constraint. Doing so would obviously improve the objective value the most. Instead of proving eqn.~(\ref{greedyProofEqn1}), it is therefore enough to show that:
\begin{align}
\Omega_{l_{T^+}} \beta_{l_{T^+}(k=1)}  &\leq  \sum_{l \in \mathcal{Q}_T} \Omega_{l} \Delta p_{l}
\label{greedyProofEqn2}
\end{align}
%
%
%
%
%
Since all $\Omega_l$'s and $\Delta p_{lk}$'s are positive, it suffices to show that: 
\begin{align}
\Omega_{l_{T^+}} \beta_{l_{T^+}(k=1)}  &\leq  \Omega_{l} \, \Delta p_{l}, \ \mbox{for any } l \in \mathcal{Q}_T ,
\label{greedyProofEqn5}
\end{align}
since, if this is true, so must be eqns.~(\ref{greedyProofEqn2}) and (\ref{greedyProofEqn1}). Let us pick some $l_T \in \mathcal{Q}_T$. Suppose that the `largest' hardening degree chosen by the greedy algorithm for $l_T$ until termination is $k_T$. Since the greedy algorithm picked $(l_T,k_T)$ before $(l_{T^+},1)$, it must be true that: 
\begin{align}
\frac{\Omega_{l_{T^+}}} {\mu_{l_{T^+}(k=1)}} &\leq \frac{\Omega_{l_{T}}} {\mu_{l_{T}k_T}}
\label{greedyProofEqn6}
\end{align}
Let $\gamma$ denote the intercept value of the $k_T^{th}$ piece of $\hat{C}_{l_T}$ on the $y$-axis. Clearly, $\gamma \leq 0$. Noting that the lower breakpoints of the first piece of all $\hat{C}_l$'s are zero, as are the values of $\hat{C}_l$ at those lower breakpoints,
\begin{align}
\frac{\Omega_{l_{T^+}}} {\mu_{l_{T^+}(k=1)}} &\leq \frac{\Omega_{l_{T}}} {\mu_{l_{T}k_T}} \\
\Rightarrow \Omega_{l_{T^+}} \left[ \frac{\beta_{l_{T^+}(k=1)}}{\hat{C}_{l_{T^+}}\left(\beta_{l_{T^+}(k=1)}\right)} \right]  &\leq \Omega_{l_{T}} \left[ \frac{\Delta p_{l_T k_T}}{\hat{C}_{l_{T}}\left(\Delta p_{l_T k_T}\right) - \gamma} \right] \\
&\leq \Omega_{l_{T}} \left[ \frac{\Delta p_{l_T k_T}}{\hat{C}_{l_{T}}\left(\Delta p_{l_T k_T}\right)} \right] \\
&:= \Omega_{l_{T}} \left[ \frac{\Delta p_{l_T}}{\hat{C}_{l_{T}}\left(\Delta p_{l_T}\right)} \right]
\label{greedyProofEqn7}
\end{align}
where the last equality follows from eqn.~(\ref{greedyAlgoVariableRecovery}) and the fact that $k_T$ is the largest hardening degree for $l_T$. By assumption, the addition-and-subtraction strategy does not cause any violation of the budget constraint, which implies that $\hat{C}_{l_{T^+}}\left(\beta_{l_{T^+}(k=1)}\right) \leq \hat{C}_{l_{T}}\left(\Delta p_{l_T}\right)$. It follows therefore from eqn.~(\ref{greedyProofEqn7}) that $\Omega_{l_{T^+}} \beta_{l_{T^+}(k=1)}  \leq  \Omega_{l_T} \, \Delta p_{l_T}$, thereby proving that inequality (\ref{greedyProofEqn5}) holds, and consequently inequality (\ref{greedyProofEqn1}). We have thus shown that it is not feasible to increase the objective value by choosing to harden an edge which has not already been chosen for some degree of hardening by the greedy algorithm until termination.

This completes the proof of optimality of the greedy algorithm.
\end{proof}

\end{document}

%% file: budget_versus_harmMILP_harmGreedy_objMILP_objGreedy.tex
\begin{tikzpicture}[x=1pt,y=1pt]
\definecolor{fillColor}{RGB}{255,255,255}
\path[use as bounding box,fill=fillColor,fill opacity=0.00] (0,0) rectangle (578.16,578.16);
\begin{scope}
\path[clip] (  0.00,  0.00) rectangle (578.16,578.16);
\definecolor{drawColor}{RGB}{255,255,255}
\definecolor{fillColor}{RGB}{255,255,255}

\path[draw=drawColor,line width= 0.6pt,line join=round,line cap=round,fill=fillColor] (  0.00,  0.00) rectangle (578.16,578.16);
\end{scope}
\begin{scope}
\path[clip] ( 33.42, 30.69) rectangle (572.66,572.66);
\definecolor{drawColor}{RGB}{190,190,190}

\path[draw=drawColor,line width= 0.3pt,line join=round] ( 33.42,147.50) --
	(572.66,147.50);

\path[draw=drawColor,line width= 0.3pt,line join=round] ( 33.42,316.52) --
	(572.66,316.52);

\path[draw=drawColor,line width= 0.3pt,line join=round] ( 33.42,485.54) --
	(572.66,485.54);

\path[draw=drawColor,line width= 0.3pt,line join=round] (119.21, 30.69) --
	(119.21,572.66);

\path[draw=drawColor,line width= 0.3pt,line join=round] (241.77, 30.69) --
	(241.77,572.66);

\path[draw=drawColor,line width= 0.3pt,line join=round] (364.32, 30.69) --
	(364.32,572.66);

\path[draw=drawColor,line width= 0.3pt,line join=round] (486.87, 30.69) --
	(486.87,572.66);
\definecolor{drawColor}{RGB}{255,255,255}

\path[draw=drawColor,line width= 0.6pt,line join=round] ( 33.42, 62.99) --
	(572.66, 62.99);

\path[draw=drawColor,line width= 0.6pt,line join=round] ( 33.42,232.01) --
	(572.66,232.01);

\path[draw=drawColor,line width= 0.6pt,line join=round] ( 33.42,401.03) --
	(572.66,401.03);

\path[draw=drawColor,line width= 0.6pt,line join=round] ( 33.42,570.04) --
	(572.66,570.04);

\path[draw=drawColor,line width= 0.6pt,line join=round] ( 57.93, 30.69) --
	( 57.93,572.66);

\path[draw=drawColor,line width= 0.6pt,line join=round] (180.49, 30.69) --
	(180.49,572.66);

\path[draw=drawColor,line width= 0.6pt,line join=round] (303.04, 30.69) --
	(303.04,572.66);

\path[draw=drawColor,line width= 0.6pt,line join=round] (425.60, 30.69) --
	(425.60,572.66);

\path[draw=drawColor,line width= 0.6pt,line join=round] (548.15, 30.69) --
	(548.15,572.66);
\definecolor{drawColor}{RGB}{248,118,109}

\path[draw=drawColor,line width= 0.6pt,line join=round] ( 57.93,548.02) --
	( 82.45,426.83) --
	(106.96,375.38) --
	(131.47,337.86) --
	(155.98,301.38) --
	(180.49,269.67) --
	(205.00,250.03) --
	(229.51,234.99) --
	(254.02,221.05) --
	(278.53,207.66) --
	(303.04,194.27) --
	(327.55,180.88) --
	(352.06,167.49) --
	(376.57,154.09) --
	(401.08,141.42) --
	(425.60,129.27) --
	(450.11,117.41) --
	(474.62,105.85) --
	(499.13, 94.68) --
	(523.64, 86.07) --
	(548.15, 77.59);
\definecolor{drawColor}{RGB}{124,174,0}

\path[draw=drawColor,line width= 0.6pt,line join=round] ( 57.93,548.02) --
	( 82.45,426.83) --
	(106.96,375.38) --
	(131.47,337.86) --
	(155.98,301.38) --
	(180.49,269.67) --
	(205.00,249.52) --
	(229.51,234.44) --
	(254.02,220.70) --
	(278.53,207.31) --
	(303.04,193.91) --
	(327.55,180.52) --
	(352.06,167.13) --
	(376.57,153.94) --
	(401.08,141.42) --
	(425.60,129.27) --
	(450.11,117.41) --
	(474.62,105.85) --
	(499.13, 94.68) --
	(523.64, 86.07) --
	(548.15, 77.59);
\definecolor{drawColor}{RGB}{0,191,196}

\path[draw=drawColor,line width= 0.6pt,line join=round] ( 57.93,496.04) --
	( 82.45,381.04) --
	(106.96,343.76) --
	(131.47,300.95) --
	(155.98,270.78) --
	(180.49,227.60) --
	(205.00,216.20) --
	(229.51,208.53) --
	(254.02,196.03) --
	(278.53,176.76) --
	(303.04,166.87) --
	(327.55,156.08) --
	(352.06,141.91) --
	(376.57,130.55) --
	(401.08,108.66) --
	(425.60,102.44) --
	(450.11, 96.00) --
	(474.62, 80.34) --
	(499.13, 68.43) --
	(523.64, 58.95) --
	(548.15, 55.32);
\definecolor{drawColor}{RGB}{199,124,255}

\path[draw=drawColor,line width= 0.6pt,line join=round] ( 57.93,498.43) --
	( 82.45,383.80) --
	(106.96,336.56) --
	(131.47,301.64) --
	(155.98,260.86) --
	(180.49,232.30) --
	(205.00,211.15) --
	(229.51,201.77) --
	(254.02,189.78) --
	(278.53,178.01) --
	(303.04,170.59) --
	(327.55,151.69) --
	(352.06,130.35) --
	(376.57,129.76) --
	(401.08,115.32) --
	(425.60,108.48) --
	(450.11, 95.89) --
	(474.62, 83.28) --
	(499.13, 69.56) --
	(523.64, 61.57) --
	(548.15, 55.79);
\definecolor{drawColor}{RGB}{248,118,109}
\definecolor{fillColor}{RGB}{248,118,109}

\path[draw=drawColor,line width= 0.4pt,line join=round,line cap=round,fill=fillColor] ( 57.93,548.02) circle (  1.96);

\path[draw=drawColor,line width= 0.4pt,line join=round,line cap=round,fill=fillColor] ( 82.45,426.83) circle (  1.96);

\path[draw=drawColor,line width= 0.4pt,line join=round,line cap=round,fill=fillColor] (106.96,375.38) circle (  1.96);

\path[draw=drawColor,line width= 0.4pt,line join=round,line cap=round,fill=fillColor] (131.47,337.86) circle (  1.96);

\path[draw=drawColor,line width= 0.4pt,line join=round,line cap=round,fill=fillColor] (155.98,301.38) circle (  1.96);

\path[draw=drawColor,line width= 0.4pt,line join=round,line cap=round,fill=fillColor] (180.49,269.67) circle (  1.96);

\path[draw=drawColor,line width= 0.4pt,line join=round,line cap=round,fill=fillColor] (205.00,250.03) circle (  1.96);

\path[draw=drawColor,line width= 0.4pt,line join=round,line cap=round,fill=fillColor] (229.51,234.99) circle (  1.96);

\path[draw=drawColor,line width= 0.4pt,line join=round,line cap=round,fill=fillColor] (254.02,221.05) circle (  1.96);

\path[draw=drawColor,line width= 0.4pt,line join=round,line cap=round,fill=fillColor] (278.53,207.66) circle (  1.96);

\path[draw=drawColor,line width= 0.4pt,line join=round,line cap=round,fill=fillColor] (303.04,194.27) circle (  1.96);

\path[draw=drawColor,line width= 0.4pt,line join=round,line cap=round,fill=fillColor] (327.55,180.88) circle (  1.96);

\path[draw=drawColor,line width= 0.4pt,line join=round,line cap=round,fill=fillColor] (352.06,167.49) circle (  1.96);

\path[draw=drawColor,line width= 0.4pt,line join=round,line cap=round,fill=fillColor] (376.57,154.09) circle (  1.96);

\path[draw=drawColor,line width= 0.4pt,line join=round,line cap=round,fill=fillColor] (401.08,141.42) circle (  1.96);

\path[draw=drawColor,line width= 0.4pt,line join=round,line cap=round,fill=fillColor] (425.60,129.27) circle (  1.96);

\path[draw=drawColor,line width= 0.4pt,line join=round,line cap=round,fill=fillColor] (450.11,117.41) circle (  1.96);

\path[draw=drawColor,line width= 0.4pt,line join=round,line cap=round,fill=fillColor] (474.62,105.85) circle (  1.96);

\path[draw=drawColor,line width= 0.4pt,line join=round,line cap=round,fill=fillColor] (499.13, 94.68) circle (  1.96);

\path[draw=drawColor,line width= 0.4pt,line join=round,line cap=round,fill=fillColor] (523.64, 86.07) circle (  1.96);

\path[draw=drawColor,line width= 0.4pt,line join=round,line cap=round,fill=fillColor] (548.15, 77.59) circle (  1.96);
\definecolor{drawColor}{RGB}{124,174,0}
\definecolor{fillColor}{RGB}{124,174,0}

\path[draw=drawColor,line width= 0.4pt,line join=round,line cap=round,fill=fillColor] ( 57.93,548.02) circle (  1.96);

\path[draw=drawColor,line width= 0.4pt,line join=round,line cap=round,fill=fillColor] ( 82.45,426.83) circle (  1.96);

\path[draw=drawColor,line width= 0.4pt,line join=round,line cap=round,fill=fillColor] (106.96,375.38) circle (  1.96);

\path[draw=drawColor,line width= 0.4pt,line join=round,line cap=round,fill=fillColor] (131.47,337.86) circle (  1.96);

\path[draw=drawColor,line width= 0.4pt,line join=round,line cap=round,fill=fillColor] (155.98,301.38) circle (  1.96);

\path[draw=drawColor,line width= 0.4pt,line join=round,line cap=round,fill=fillColor] (180.49,269.67) circle (  1.96);

\path[draw=drawColor,line width= 0.4pt,line join=round,line cap=round,fill=fillColor] (205.00,249.52) circle (  1.96);

\path[draw=drawColor,line width= 0.4pt,line join=round,line cap=round,fill=fillColor] (229.51,234.44) circle (  1.96);

\path[draw=drawColor,line width= 0.4pt,line join=round,line cap=round,fill=fillColor] (254.02,220.70) circle (  1.96);

\path[draw=drawColor,line width= 0.4pt,line join=round,line cap=round,fill=fillColor] (278.53,207.31) circle (  1.96);

\path[draw=drawColor,line width= 0.4pt,line join=round,line cap=round,fill=fillColor] (303.04,193.91) circle (  1.96);

\path[draw=drawColor,line width= 0.4pt,line join=round,line cap=round,fill=fillColor] (327.55,180.52) circle (  1.96);

\path[draw=drawColor,line width= 0.4pt,line join=round,line cap=round,fill=fillColor] (352.06,167.13) circle (  1.96);

\path[draw=drawColor,line width= 0.4pt,line join=round,line cap=round,fill=fillColor] (376.57,153.94) circle (  1.96);

\path[draw=drawColor,line width= 0.4pt,line join=round,line cap=round,fill=fillColor] (401.08,141.42) circle (  1.96);

\path[draw=drawColor,line width= 0.4pt,line join=round,line cap=round,fill=fillColor] (425.60,129.27) circle (  1.96);

\path[draw=drawColor,line width= 0.4pt,line join=round,line cap=round,fill=fillColor] (450.11,117.41) circle (  1.96);

\path[draw=drawColor,line width= 0.4pt,line join=round,line cap=round,fill=fillColor] (474.62,105.85) circle (  1.96);

\path[draw=drawColor,line width= 0.4pt,line join=round,line cap=round,fill=fillColor] (499.13, 94.68) circle (  1.96);

\path[draw=drawColor,line width= 0.4pt,line join=round,line cap=round,fill=fillColor] (523.64, 86.07) circle (  1.96);

\path[draw=drawColor,line width= 0.4pt,line join=round,line cap=round,fill=fillColor] (548.15, 77.59) circle (  1.96);
\definecolor{drawColor}{RGB}{0,191,196}
\definecolor{fillColor}{RGB}{0,191,196}

\path[draw=drawColor,line width= 0.4pt,line join=round,line cap=round,fill=fillColor] ( 57.93,496.04) circle (  1.96);

\path[draw=drawColor,line width= 0.4pt,line join=round,line cap=round,fill=fillColor] ( 82.45,381.04) circle (  1.96);

\path[draw=drawColor,line width= 0.4pt,line join=round,line cap=round,fill=fillColor] (106.96,343.76) circle (  1.96);

\path[draw=drawColor,line width= 0.4pt,line join=round,line cap=round,fill=fillColor] (131.47,300.95) circle (  1.96);

\path[draw=drawColor,line width= 0.4pt,line join=round,line cap=round,fill=fillColor] (155.98,270.78) circle (  1.96);

\path[draw=drawColor,line width= 0.4pt,line join=round,line cap=round,fill=fillColor] (180.49,227.60) circle (  1.96);

\path[draw=drawColor,line width= 0.4pt,line join=round,line cap=round,fill=fillColor] (205.00,216.20) circle (  1.96);

\path[draw=drawColor,line width= 0.4pt,line join=round,line cap=round,fill=fillColor] (229.51,208.53) circle (  1.96);

\path[draw=drawColor,line width= 0.4pt,line join=round,line cap=round,fill=fillColor] (254.02,196.03) circle (  1.96);

\path[draw=drawColor,line width= 0.4pt,line join=round,line cap=round,fill=fillColor] (278.53,176.76) circle (  1.96);

\path[draw=drawColor,line width= 0.4pt,line join=round,line cap=round,fill=fillColor] (303.04,166.87) circle (  1.96);

\path[draw=drawColor,line width= 0.4pt,line join=round,line cap=round,fill=fillColor] (327.55,156.08) circle (  1.96);

\path[draw=drawColor,line width= 0.4pt,line join=round,line cap=round,fill=fillColor] (352.06,141.91) circle (  1.96);

\path[draw=drawColor,line width= 0.4pt,line join=round,line cap=round,fill=fillColor] (376.57,130.55) circle (  1.96);

\path[draw=drawColor,line width= 0.4pt,line join=round,line cap=round,fill=fillColor] (401.08,108.66) circle (  1.96);

\path[draw=drawColor,line width= 0.4pt,line join=round,line cap=round,fill=fillColor] (425.60,102.44) circle (  1.96);

\path[draw=drawColor,line width= 0.4pt,line join=round,line cap=round,fill=fillColor] (450.11, 96.00) circle (  1.96);

\path[draw=drawColor,line width= 0.4pt,line join=round,line cap=round,fill=fillColor] (474.62, 80.34) circle (  1.96);

\path[draw=drawColor,line width= 0.4pt,line join=round,line cap=round,fill=fillColor] (499.13, 68.43) circle (  1.96);

\path[draw=drawColor,line width= 0.4pt,line join=round,line cap=round,fill=fillColor] (523.64, 58.95) circle (  1.96);

\path[draw=drawColor,line width= 0.4pt,line join=round,line cap=round,fill=fillColor] (548.15, 55.32) circle (  1.96);
\definecolor{drawColor}{RGB}{199,124,255}
\definecolor{fillColor}{RGB}{199,124,255}

\path[draw=drawColor,line width= 0.4pt,line join=round,line cap=round,fill=fillColor] ( 57.93,498.43) circle (  1.96);

\path[draw=drawColor,line width= 0.4pt,line join=round,line cap=round,fill=fillColor] ( 82.45,383.80) circle (  1.96);

\path[draw=drawColor,line width= 0.4pt,line join=round,line cap=round,fill=fillColor] (106.96,336.56) circle (  1.96);

\path[draw=drawColor,line width= 0.4pt,line join=round,line cap=round,fill=fillColor] (131.47,301.64) circle (  1.96);

\path[draw=drawColor,line width= 0.4pt,line join=round,line cap=round,fill=fillColor] (155.98,260.86) circle (  1.96);

\path[draw=drawColor,line width= 0.4pt,line join=round,line cap=round,fill=fillColor] (180.49,232.30) circle (  1.96);

\path[draw=drawColor,line width= 0.4pt,line join=round,line cap=round,fill=fillColor] (205.00,211.15) circle (  1.96);

\path[draw=drawColor,line width= 0.4pt,line join=round,line cap=round,fill=fillColor] (229.51,201.77) circle (  1.96);

\path[draw=drawColor,line width= 0.4pt,line join=round,line cap=round,fill=fillColor] (254.02,189.78) circle (  1.96);

\path[draw=drawColor,line width= 0.4pt,line join=round,line cap=round,fill=fillColor] (278.53,178.01) circle (  1.96);

\path[draw=drawColor,line width= 0.4pt,line join=round,line cap=round,fill=fillColor] (303.04,170.59) circle (  1.96);

\path[draw=drawColor,line width= 0.4pt,line join=round,line cap=round,fill=fillColor] (327.55,151.69) circle (  1.96);

\path[draw=drawColor,line width= 0.4pt,line join=round,line cap=round,fill=fillColor] (352.06,130.35) circle (  1.96);

\path[draw=drawColor,line width= 0.4pt,line join=round,line cap=round,fill=fillColor] (376.57,129.76) circle (  1.96);

\path[draw=drawColor,line width= 0.4pt,line join=round,line cap=round,fill=fillColor] (401.08,115.32) circle (  1.96);

\path[draw=drawColor,line width= 0.4pt,line join=round,line cap=round,fill=fillColor] (425.60,108.48) circle (  1.96);

\path[draw=drawColor,line width= 0.4pt,line join=round,line cap=round,fill=fillColor] (450.11, 95.89) circle (  1.96);

\path[draw=drawColor,line width= 0.4pt,line join=round,line cap=round,fill=fillColor] (474.62, 83.28) circle (  1.96);

\path[draw=drawColor,line width= 0.4pt,line join=round,line cap=round,fill=fillColor] (499.13, 69.56) circle (  1.96);

\path[draw=drawColor,line width= 0.4pt,line join=round,line cap=round,fill=fillColor] (523.64, 61.57) circle (  1.96);

\path[draw=drawColor,line width= 0.4pt,line join=round,line cap=round,fill=fillColor] (548.15, 55.79) circle (  1.96);
\end{scope}
\begin{scope}
\path[clip] (  0.00,  0.00) rectangle (578.16,578.16);
\definecolor{drawColor}{RGB}{0,0,0}

\path[draw=drawColor,line width= 0.6pt,line join=round] ( 33.42, 30.69) --
	( 33.42,572.66);
\end{scope}
\begin{scope}
\path[clip] (  0.00,  0.00) rectangle (578.16,578.16);
\definecolor{drawColor}{gray}{0.30}

\node[text=drawColor,anchor=base east,inner sep=0pt, outer sep=0pt, scale=  0.88] at ( 28.47, 59.96) {12};

\node[text=drawColor,anchor=base east,inner sep=0pt, outer sep=0pt, scale=  0.88] at ( 28.47,228.98) {14};

\node[text=drawColor,anchor=base east,inner sep=0pt, outer sep=0pt, scale=  0.88] at ( 28.47,398.00) {16};

\node[text=drawColor,anchor=base east,inner sep=0pt, outer sep=0pt, scale=  0.88] at ( 28.47,567.01) {18};
\end{scope}
\begin{scope}
\path[clip] (  0.00,  0.00) rectangle (578.16,578.16);
\definecolor{drawColor}{gray}{0.20}

\path[draw=drawColor,line width= 0.6pt,line join=round] ( 30.67, 62.99) --
	( 33.42, 62.99);

\path[draw=drawColor,line width= 0.6pt,line join=round] ( 30.67,232.01) --
	( 33.42,232.01);

\path[draw=drawColor,line width= 0.6pt,line join=round] ( 30.67,401.03) --
	( 33.42,401.03);

\path[draw=drawColor,line width= 0.6pt,line join=round] ( 30.67,570.04) --
	( 33.42,570.04);
\end{scope}
\begin{scope}
\path[clip] (  0.00,  0.00) rectangle (578.16,578.16);
\definecolor{drawColor}{RGB}{0,0,0}

\path[draw=drawColor,line width= 0.6pt,line join=round] ( 33.42, 30.69) --
	(572.66, 30.69);
\end{scope}
\begin{scope}
\path[clip] (  0.00,  0.00) rectangle (578.16,578.16);
\definecolor{drawColor}{gray}{0.20}

\path[draw=drawColor,line width= 0.6pt,line join=round] ( 57.93, 27.94) --
	( 57.93, 30.69);

\path[draw=drawColor,line width= 0.6pt,line join=round] (180.49, 27.94) --
	(180.49, 30.69);

\path[draw=drawColor,line width= 0.6pt,line join=round] (303.04, 27.94) --
	(303.04, 30.69);

\path[draw=drawColor,line width= 0.6pt,line join=round] (425.60, 27.94) --
	(425.60, 30.69);

\path[draw=drawColor,line width= 0.6pt,line join=round] (548.15, 27.94) --
	(548.15, 30.69);
\end{scope}
\begin{scope}
\path[clip] (  0.00,  0.00) rectangle (578.16,578.16);
\definecolor{drawColor}{gray}{0.30}

\node[text=drawColor,anchor=base,inner sep=0pt, outer sep=0pt, scale=  0.88] at ( 57.93, 19.68) {0};

\node[text=drawColor,anchor=base,inner sep=0pt, outer sep=0pt, scale=  0.88] at (180.49, 19.68) {5};

\node[text=drawColor,anchor=base,inner sep=0pt, outer sep=0pt, scale=  0.88] at (303.04, 19.68) {10};

\node[text=drawColor,anchor=base,inner sep=0pt, outer sep=0pt, scale=  0.88] at (425.60, 19.68) {15};

\node[text=drawColor,anchor=base,inner sep=0pt, outer sep=0pt, scale=  0.88] at (548.15, 19.68) {20};
\end{scope}
\begin{scope}
\path[clip] (  0.00,  0.00) rectangle (578.16,578.16);
\definecolor{drawColor}{RGB}{0,0,0}

\node[text=drawColor,anchor=base,inner sep=0pt, outer sep=0pt, scale=  1.10] at (303.04,  7.70) {Budget};
\end{scope}
\begin{scope}
\path[clip] (  0.00,  0.00) rectangle (578.16,578.16);
\definecolor{drawColor}{RGB}{0,0,0}

\node[text=drawColor,rotate= 90.00,anchor=base,inner sep=0pt, outer sep=0pt, scale=  1.10] at ( 15.28,301.67) {Aggregate harm};
\end{scope}
\begin{scope}
\path[clip] (  0.00,  0.00) rectangle (578.16,578.16);
\definecolor{fillColor}{RGB}{255,255,255}

\path[fill=fillColor] (443.08,447.60) rectangle (537.16,537.03);
\end{scope}
\begin{scope}
\path[clip] (  0.00,  0.00) rectangle (578.16,578.16);
\definecolor{drawColor}{RGB}{0,0,0}

\node[text=drawColor,anchor=base west,inner sep=0pt, outer sep=0pt, scale=  1.10] at (447.34,525.18) {Objectives by};

\node[text=drawColor,anchor=base west,inner sep=0pt, outer sep=0pt, scale=  1.10] at (447.34,513.30) {algorithms};
\end{scope}
\begin{scope}
\path[clip] (  0.00,  0.00) rectangle (578.16,578.16);
\definecolor{drawColor}{RGB}{255,255,255}
\definecolor{fillColor}{gray}{0.95}

\path[draw=drawColor,line width= 0.6pt,line join=round,line cap=round,fill=fillColor] (447.34,495.23) rectangle (461.80,509.69);
\end{scope}
\begin{scope}
\path[clip] (  0.00,  0.00) rectangle (578.16,578.16);
\definecolor{drawColor}{RGB}{248,118,109}

\path[draw=drawColor,line width= 0.6pt,line join=round] (448.79,502.46) -- (460.35,502.46);
\end{scope}
\begin{scope}
\path[clip] (  0.00,  0.00) rectangle (578.16,578.16);
\definecolor{drawColor}{RGB}{248,118,109}
\definecolor{fillColor}{RGB}{248,118,109}

\path[draw=drawColor,line width= 0.4pt,line join=round,line cap=round,fill=fillColor] (454.57,502.46) circle (  1.96);
\end{scope}
\begin{scope}
\path[clip] (  0.00,  0.00) rectangle (578.16,578.16);
\definecolor{drawColor}{RGB}{255,255,255}
\definecolor{fillColor}{gray}{0.95}

\path[draw=drawColor,line width= 0.6pt,line join=round,line cap=round,fill=fillColor] (447.34,480.78) rectangle (461.80,495.23);
\end{scope}
\begin{scope}
\path[clip] (  0.00,  0.00) rectangle (578.16,578.16);
\definecolor{drawColor}{RGB}{124,174,0}

\path[draw=drawColor,line width= 0.6pt,line join=round] (448.79,488.01) -- (460.35,488.01);
\end{scope}
\begin{scope}
\path[clip] (  0.00,  0.00) rectangle (578.16,578.16);
\definecolor{drawColor}{RGB}{124,174,0}
\definecolor{fillColor}{RGB}{124,174,0}

\path[draw=drawColor,line width= 0.4pt,line join=round,line cap=round,fill=fillColor] (454.57,488.01) circle (  1.96);
\end{scope}
\begin{scope}
\path[clip] (  0.00,  0.00) rectangle (578.16,578.16);
\definecolor{drawColor}{RGB}{255,255,255}
\definecolor{fillColor}{gray}{0.95}

\path[draw=drawColor,line width= 0.6pt,line join=round,line cap=round,fill=fillColor] (447.34,466.33) rectangle (461.80,480.78);
\end{scope}
\begin{scope}
\path[clip] (  0.00,  0.00) rectangle (578.16,578.16);
\definecolor{drawColor}{RGB}{0,191,196}

\path[draw=drawColor,line width= 0.6pt,line join=round] (448.79,473.55) -- (460.35,473.55);
\end{scope}
\begin{scope}
\path[clip] (  0.00,  0.00) rectangle (578.16,578.16);
\definecolor{drawColor}{RGB}{0,191,196}
\definecolor{fillColor}{RGB}{0,191,196}

\path[draw=drawColor,line width= 0.4pt,line join=round,line cap=round,fill=fillColor] (454.57,473.55) circle (  1.96);
\end{scope}
\begin{scope}
\path[clip] (  0.00,  0.00) rectangle (578.16,578.16);
\definecolor{drawColor}{RGB}{255,255,255}
\definecolor{fillColor}{gray}{0.95}

\path[draw=drawColor,line width= 0.6pt,line join=round,line cap=round,fill=fillColor] (447.34,451.87) rectangle (461.80,466.33);
\end{scope}
\begin{scope}
\path[clip] (  0.00,  0.00) rectangle (578.16,578.16);
\definecolor{drawColor}{RGB}{199,124,255}

\path[draw=drawColor,line width= 0.6pt,line join=round] (448.79,459.10) -- (460.35,459.10);
\end{scope}
\begin{scope}
\path[clip] (  0.00,  0.00) rectangle (578.16,578.16);
\definecolor{drawColor}{RGB}{199,124,255}
\definecolor{fillColor}{RGB}{199,124,255}

\path[draw=drawColor,line width= 0.4pt,line join=round,line cap=round,fill=fillColor] (454.57,459.10) circle (  1.96);
\end{scope}
\begin{scope}
\path[clip] (  0.00,  0.00) rectangle (578.16,578.16);
\definecolor{drawColor}{RGB}{0,0,0}

\node[text=drawColor,anchor=base west,inner sep=0pt, outer sep=0pt, scale=  0.88] at (463.60,499.43) {$f(E[\cdot])$ by Greedy};
\end{scope}
\begin{scope}
\path[clip] (  0.00,  0.00) rectangle (578.16,578.16);
\definecolor{drawColor}{RGB}{0,0,0}

\node[text=drawColor,anchor=base west,inner sep=0pt, outer sep=0pt, scale=  0.88] at (463.60,484.98) {$f(E[\cdot])$ by MILP};
\end{scope}
\begin{scope}
\path[clip] (  0.00,  0.00) rectangle (578.16,578.16);
\definecolor{drawColor}{RGB}{0,0,0}

\node[text=drawColor,anchor=base west,inner sep=0pt, outer sep=0pt, scale=  0.88] at (463.60,470.52) {$E[f(\cdot)]$ by Greedy};
\end{scope}
\begin{scope}
\path[clip] (  0.00,  0.00) rectangle (578.16,578.16);
\definecolor{drawColor}{RGB}{0,0,0}

\node[text=drawColor,anchor=base west,inner sep=0pt, outer sep=0pt, scale=  0.88] at (463.60,456.07) {$E[f(\cdot)]$ by MILP};
\end{scope}
\end{tikzpicture}

%% file: greedyharm8500.tex
\begin{tikzpicture}[x=1pt,y=1pt]
\definecolor{fillColor}{RGB}{255,255,255}
\path[use as bounding box,fill=fillColor,fill opacity=0.00] (0,0) rectangle (578.16,433.62);
\begin{scope}
\path[clip] (  0.00,  0.00) rectangle (578.16,433.62);
\definecolor{drawColor}{RGB}{255,255,255}
\definecolor{fillColor}{RGB}{255,255,255}

\path[draw=drawColor,line width= 0.6pt,line join=round,line cap=round,fill=fillColor] (  0.00,  0.00) rectangle (578.16,433.62);
\end{scope}
\begin{scope}
\path[clip] ( 59.82, 30.69) rectangle (572.66,428.12);
\definecolor{drawColor}{RGB}{190,190,190}

\path[draw=drawColor,line width= 0.3pt,line join=round] ( 59.82, 72.13) --
	(572.66, 72.13);

\path[draw=drawColor,line width= 0.3pt,line join=round] ( 59.82,152.29) --
	(572.66,152.29);

\path[draw=drawColor,line width= 0.3pt,line join=round] ( 59.82,232.44) --
	(572.66,232.44);

\path[draw=drawColor,line width= 0.3pt,line join=round] ( 59.82,312.59) --
	(572.66,312.59);

\path[draw=drawColor,line width= 0.3pt,line join=round] ( 59.82,392.75) --
	(572.66,392.75);

\path[draw=drawColor,line width= 0.3pt,line join=round] (160.83, 30.69) --
	(160.83,428.12);

\path[draw=drawColor,line width= 0.3pt,line join=round] (316.24, 30.69) --
	(316.24,428.12);

\path[draw=drawColor,line width= 0.3pt,line join=round] (471.65, 30.69) --
	(471.65,428.12);
\definecolor{drawColor}{RGB}{255,255,255}

\path[draw=drawColor,line width= 0.6pt,line join=round] ( 59.82, 32.06) --
	(572.66, 32.06);

\path[draw=drawColor,line width= 0.6pt,line join=round] ( 59.82,112.21) --
	(572.66,112.21);

\path[draw=drawColor,line width= 0.6pt,line join=round] ( 59.82,192.36) --
	(572.66,192.36);

\path[draw=drawColor,line width= 0.6pt,line join=round] ( 59.82,272.52) --
	(572.66,272.52);

\path[draw=drawColor,line width= 0.6pt,line join=round] ( 59.82,352.67) --
	(572.66,352.67);

\path[draw=drawColor,line width= 0.6pt,line join=round] ( 83.13, 30.69) --
	( 83.13,428.12);

\path[draw=drawColor,line width= 0.6pt,line join=round] (238.54, 30.69) --
	(238.54,428.12);

\path[draw=drawColor,line width= 0.6pt,line join=round] (393.94, 30.69) --
	(393.94,428.12);

\path[draw=drawColor,line width= 0.6pt,line join=round] (549.35, 30.69) --
	(549.35,428.12);
\definecolor{drawColor}{RGB}{0,0,0}

\path[draw=drawColor,line width= 0.6pt,line join=round] ( 83.13,410.05) --
	( 98.67,386.20) --
	(114.21,364.42) --
	(129.75,344.16) --
	(145.29,325.73) --
	(160.83,309.30) --
	(176.37,293.56) --
	(191.91,278.64) --
	(207.45,264.35) --
	(222.99,250.87) --
	(238.54,238.61) --
	(254.08,226.69) --
	(269.62,215.18) --
	(285.16,203.77) --
	(300.70,192.95) --
	(316.24,182.56) --
	(331.78,172.50) --
	(347.32,162.69) --
	(362.86,152.85) --
	(378.40,143.25) --
	(393.94,133.88) --
	(409.48,124.57) --
	(425.02,115.45) --
	(440.56,106.56) --
	(456.10, 97.88) --
	(471.65, 89.39) --
	(487.19, 81.09) --
	(502.73, 72.80) --
	(518.27, 64.71) --
	(533.81, 56.67) --
	(549.35, 48.75);
\definecolor{fillColor}{RGB}{0,0,0}

\path[draw=drawColor,line width= 0.4pt,line join=round,line cap=round,fill=fillColor] ( 83.13,410.05) circle (  1.96);

\path[draw=drawColor,line width= 0.4pt,line join=round,line cap=round,fill=fillColor] ( 98.67,386.20) circle (  1.96);

\path[draw=drawColor,line width= 0.4pt,line join=round,line cap=round,fill=fillColor] (114.21,364.42) circle (  1.96);

\path[draw=drawColor,line width= 0.4pt,line join=round,line cap=round,fill=fillColor] (129.75,344.16) circle (  1.96);

\path[draw=drawColor,line width= 0.4pt,line join=round,line cap=round,fill=fillColor] (145.29,325.73) circle (  1.96);

\path[draw=drawColor,line width= 0.4pt,line join=round,line cap=round,fill=fillColor] (160.83,309.30) circle (  1.96);

\path[draw=drawColor,line width= 0.4pt,line join=round,line cap=round,fill=fillColor] (176.37,293.56) circle (  1.96);

\path[draw=drawColor,line width= 0.4pt,line join=round,line cap=round,fill=fillColor] (191.91,278.64) circle (  1.96);

\path[draw=drawColor,line width= 0.4pt,line join=round,line cap=round,fill=fillColor] (207.45,264.35) circle (  1.96);

\path[draw=drawColor,line width= 0.4pt,line join=round,line cap=round,fill=fillColor] (222.99,250.87) circle (  1.96);

\path[draw=drawColor,line width= 0.4pt,line join=round,line cap=round,fill=fillColor] (238.54,238.61) circle (  1.96);

\path[draw=drawColor,line width= 0.4pt,line join=round,line cap=round,fill=fillColor] (254.08,226.69) circle (  1.96);

\path[draw=drawColor,line width= 0.4pt,line join=round,line cap=round,fill=fillColor] (269.62,215.18) circle (  1.96);

\path[draw=drawColor,line width= 0.4pt,line join=round,line cap=round,fill=fillColor] (285.16,203.77) circle (  1.96);

\path[draw=drawColor,line width= 0.4pt,line join=round,line cap=round,fill=fillColor] (300.70,192.95) circle (  1.96);

\path[draw=drawColor,line width= 0.4pt,line join=round,line cap=round,fill=fillColor] (316.24,182.56) circle (  1.96);

\path[draw=drawColor,line width= 0.4pt,line join=round,line cap=round,fill=fillColor] (331.78,172.50) circle (  1.96);

\path[draw=drawColor,line width= 0.4pt,line join=round,line cap=round,fill=fillColor] (347.32,162.69) circle (  1.96);

\path[draw=drawColor,line width= 0.4pt,line join=round,line cap=round,fill=fillColor] (362.86,152.85) circle (  1.96);

\path[draw=drawColor,line width= 0.4pt,line join=round,line cap=round,fill=fillColor] (378.40,143.25) circle (  1.96);

\path[draw=drawColor,line width= 0.4pt,line join=round,line cap=round,fill=fillColor] (393.94,133.88) circle (  1.96);

\path[draw=drawColor,line width= 0.4pt,line join=round,line cap=round,fill=fillColor] (409.48,124.57) circle (  1.96);

\path[draw=drawColor,line width= 0.4pt,line join=round,line cap=round,fill=fillColor] (425.02,115.45) circle (  1.96);

\path[draw=drawColor,line width= 0.4pt,line join=round,line cap=round,fill=fillColor] (440.56,106.56) circle (  1.96);

\path[draw=drawColor,line width= 0.4pt,line join=round,line cap=round,fill=fillColor] (456.10, 97.88) circle (  1.96);

\path[draw=drawColor,line width= 0.4pt,line join=round,line cap=round,fill=fillColor] (471.65, 89.39) circle (  1.96);

\path[draw=drawColor,line width= 0.4pt,line join=round,line cap=round,fill=fillColor] (487.19, 81.09) circle (  1.96);

\path[draw=drawColor,line width= 0.4pt,line join=round,line cap=round,fill=fillColor] (502.73, 72.80) circle (  1.96);

\path[draw=drawColor,line width= 0.4pt,line join=round,line cap=round,fill=fillColor] (518.27, 64.71) circle (  1.96);

\path[draw=drawColor,line width= 0.4pt,line join=round,line cap=round,fill=fillColor] (533.81, 56.67) circle (  1.96);

\path[draw=drawColor,line width= 0.4pt,line join=round,line cap=round,fill=fillColor] (549.35, 48.75) circle (  1.96);
\end{scope}
\begin{scope}
\path[clip] (  0.00,  0.00) rectangle (578.16,433.62);
\definecolor{drawColor}{RGB}{0,0,0}

\path[draw=drawColor,line width= 0.6pt,line join=round] ( 59.82, 30.69) --
	( 59.82,428.12);
\end{scope}
\begin{scope}
\path[clip] (  0.00,  0.00) rectangle (578.16,433.62);
\definecolor{drawColor}{gray}{0.30}

\node[text=drawColor,anchor=base east,inner sep=0pt, outer sep=0pt, scale=  0.88] at ( 54.87, 29.02) {6.88e+06};

\node[text=drawColor,anchor=base east,inner sep=0pt, outer sep=0pt, scale=  0.88] at ( 54.87,109.18) {6.92e+06};

\node[text=drawColor,anchor=base east,inner sep=0pt, outer sep=0pt, scale=  0.88] at ( 54.87,189.33) {6.96e+06};

\node[text=drawColor,anchor=base east,inner sep=0pt, outer sep=0pt, scale=  0.88] at ( 54.87,269.49) {7.00e+06};

\node[text=drawColor,anchor=base east,inner sep=0pt, outer sep=0pt, scale=  0.88] at ( 54.87,349.64) {7.04e+06};
\end{scope}
\begin{scope}
\path[clip] (  0.00,  0.00) rectangle (578.16,433.62);
\definecolor{drawColor}{gray}{0.20}

\path[draw=drawColor,line width= 0.6pt,line join=round] ( 57.07, 32.06) --
	( 59.82, 32.06);

\path[draw=drawColor,line width= 0.6pt,line join=round] ( 57.07,112.21) --
	( 59.82,112.21);

\path[draw=drawColor,line width= 0.6pt,line join=round] ( 57.07,192.36) --
	( 59.82,192.36);

\path[draw=drawColor,line width= 0.6pt,line join=round] ( 57.07,272.52) --
	( 59.82,272.52);

\path[draw=drawColor,line width= 0.6pt,line join=round] ( 57.07,352.67) --
	( 59.82,352.67);
\end{scope}
\begin{scope}
\path[clip] (  0.00,  0.00) rectangle (578.16,433.62);
\definecolor{drawColor}{RGB}{0,0,0}

\path[draw=drawColor,line width= 0.6pt,line join=round] ( 59.82, 30.69) --
	(572.66, 30.69);
\end{scope}
\begin{scope}
\path[clip] (  0.00,  0.00) rectangle (578.16,433.62);
\definecolor{drawColor}{gray}{0.20}

\path[draw=drawColor,line width= 0.6pt,line join=round] ( 83.13, 27.94) --
	( 83.13, 30.69);

\path[draw=drawColor,line width= 0.6pt,line join=round] (238.54, 27.94) --
	(238.54, 30.69);

\path[draw=drawColor,line width= 0.6pt,line join=round] (393.94, 27.94) --
	(393.94, 30.69);

\path[draw=drawColor,line width= 0.6pt,line join=round] (549.35, 27.94) --
	(549.35, 30.69);
\end{scope}
\begin{scope}
\path[clip] (  0.00,  0.00) rectangle (578.16,433.62);
\definecolor{drawColor}{gray}{0.30}

\node[text=drawColor,anchor=base,inner sep=0pt, outer sep=0pt, scale=  0.88] at ( 83.13, 19.68) {0};

\node[text=drawColor,anchor=base,inner sep=0pt, outer sep=0pt, scale=  0.88] at (238.54, 19.68) {20};

\node[text=drawColor,anchor=base,inner sep=0pt, outer sep=0pt, scale=  0.88] at (393.94, 19.68) {40};

\node[text=drawColor,anchor=base,inner sep=0pt, outer sep=0pt, scale=  0.88] at (549.35, 19.68) {60};
\end{scope}
\begin{scope}
\path[clip] (  0.00,  0.00) rectangle (578.16,433.62);
\definecolor{drawColor}{RGB}{0,0,0}

\node[text=drawColor,anchor=base,inner sep=0pt, outer sep=0pt, scale=  1.10] at (316.24,  7.70) {Budget};
\end{scope}
\begin{scope}
\path[clip] (  0.00,  0.00) rectangle (578.16,433.62);
\definecolor{drawColor}{RGB}{0,0,0}

\node[text=drawColor,rotate= 90.00,anchor=base,inner sep=0pt, outer sep=0pt, scale=  1.10] at ( 15.28,229.40) {$f(E[\cdot])$ by greedy algorithm};
\end{scope}
\end{tikzpicture}